\documentclass{article}

\usepackage{amsmath, amssymb, amsthm,amsfonts}
\usepackage{mathrsfs}
\usepackage{color}
\usepackage{cite}
\usepackage{graphicx}
\usepackage[colorlinks,linkcolor=red,anchorcolor=blue,citecolor=blue]{hyperref}

\usepackage{enumerate}

\usepackage[all]{xy}

\newtheorem{definition}{Definition}[section]
\newtheorem{theorem}{Theorem}[section]

\newtheorem{lemma}{Lemma}[section]
\newtheorem{proposition}{Proposition}[section]
\newtheorem{corollary}{Corollary}[section]
\newtheorem{remark}{Remark}[section]

\author{Weiyuan Qiu and Lingrui Wang}
\date{March 9, 2026}

\title{Hyperbolic components of cosine family with a fixed critical point}

\begin{document}

\maketitle

\begin{abstract}
    We studied the parameter plane of the cosine functions with a fixed critical point. The hyperbolic components can be classified into three types: A, C and D. All the hyperbolic components are bounded and simply connected, except for the unique type-A component, which contains $0$ as an isolated boundary point. Using the method of para-puzzle, we constructed a phase-parameter transfer mapping and proved that the boundaries of hyperbolic components are Jordan curves. By a similar idea, the hyperbolic components of type C are quasidisks. 

\end{abstract}

\section{Introduction}\label{sec: Introduction}
In the study of parameter spaces of analytic functions, Fatou conjectured in \cite{Fat20} that hyperbolic rational maps are dense in the space of rational maps with given degree. It was also conjectured for the family of exponential functions and transcendental entire functions with finitely many singular values in \cite{EL92}. This conjecture is known as `The Density of Hyperbolicity'. It was proved in the family of real polynomials (\cite{Lyu97,GS98} for quadratic polynomials and \cite{KSS07} for polynomials of all degrees). 

In the family of quadratic polynomials, the density of hyperbolicity can be deduced from another conjecture: the boundary of the Mandelbrot set $\mathcal{M}$ is locally connected (\cite{DH84}). The local connectivity of the boundary of $\mathcal{M}$ is known as the MLC conjecture. According to Whyburn's characterization of local connectivity (\cite{Why42}), to study the MLC conjecture, it is necessary to research the local connectivity of the boundaries of hyperbolic components. 

The hyperbolic components of many families of rational functions and polynomials have been studied for decades (see \cite{BH88,BH92,QRWY15,Ree90,Ree92,Mil12,Roe07,Fau92, Wan21}). As for the parameter spaces of transcendental functions, the exponential family $\{f_{\kappa}(z)=e^z+\kappa: \kappa \in \mathbb{C}\}$ is well-understood. The hyperbolic components of the exponential family are classified by kneading sequences (\cite{DFJ02}), and the boundaries of hyperbolic components are all Jordan arcs tending to $\infty$ (\cite{RS09}). The parameter space of tangent family $\{T_{\lambda}(z)=\lambda \tan z: \lambda \neq 0\}$ is also studied in \cite{KK97, KK00, KY06, Yua12}. However, the research of other transcendental families is rare. 

In this paper, we consider the family $\{f_{a,b}(z)=ae^z+be^{-z}: a,b \neq 0\}$ with two parameters, which is called the cosine family following \cite{McM87}. The cosine function is one of the simplest transcendental entire function in the following sense: it has no finite asymptotic values and two critical values $\pm 2\sqrt{ab}$, and the critical points $\log \sqrt{b/a}+k\pi i$, $k \in \mathbb{Z}$, are of $\pi i$-period (with branch of the logarithm and the square root suitably chosen). As in Milnor's research (\cite{Mil09}) of the parameter space of cubic polynomials, we consider the sections $\mathcal{S}_p$, $p \ge 1$ first. Here $p$ is an integer, and 
$$\mathcal{S}_p=\{f_{a,b}: f_{a,b} \text{ has a superattracting cycle of periodic } p\}.$$
In particular, we consider the section $\mathcal{S}_1$, in which $f_{a,b}$ has a superattracting fixed point. Up to a conjugacy, maps in $\mathcal{S}_1$ have the following form: 
$$f_v(z)=v(\cos z-1), \quad v \neq 0.$$
Thus $\mathcal{S}_1$ can be identified with the punctured plane $\mathbb{C}^*$ (Figure \ref{fig: parameter space of f_v}). By abuse of notations, we do not distinguish the space of $f_v$ and the $v$-plane. 

\begin{figure}[h]\label{fig: parameter space of f_v}
    \centering
    \includegraphics[scale=0.45]{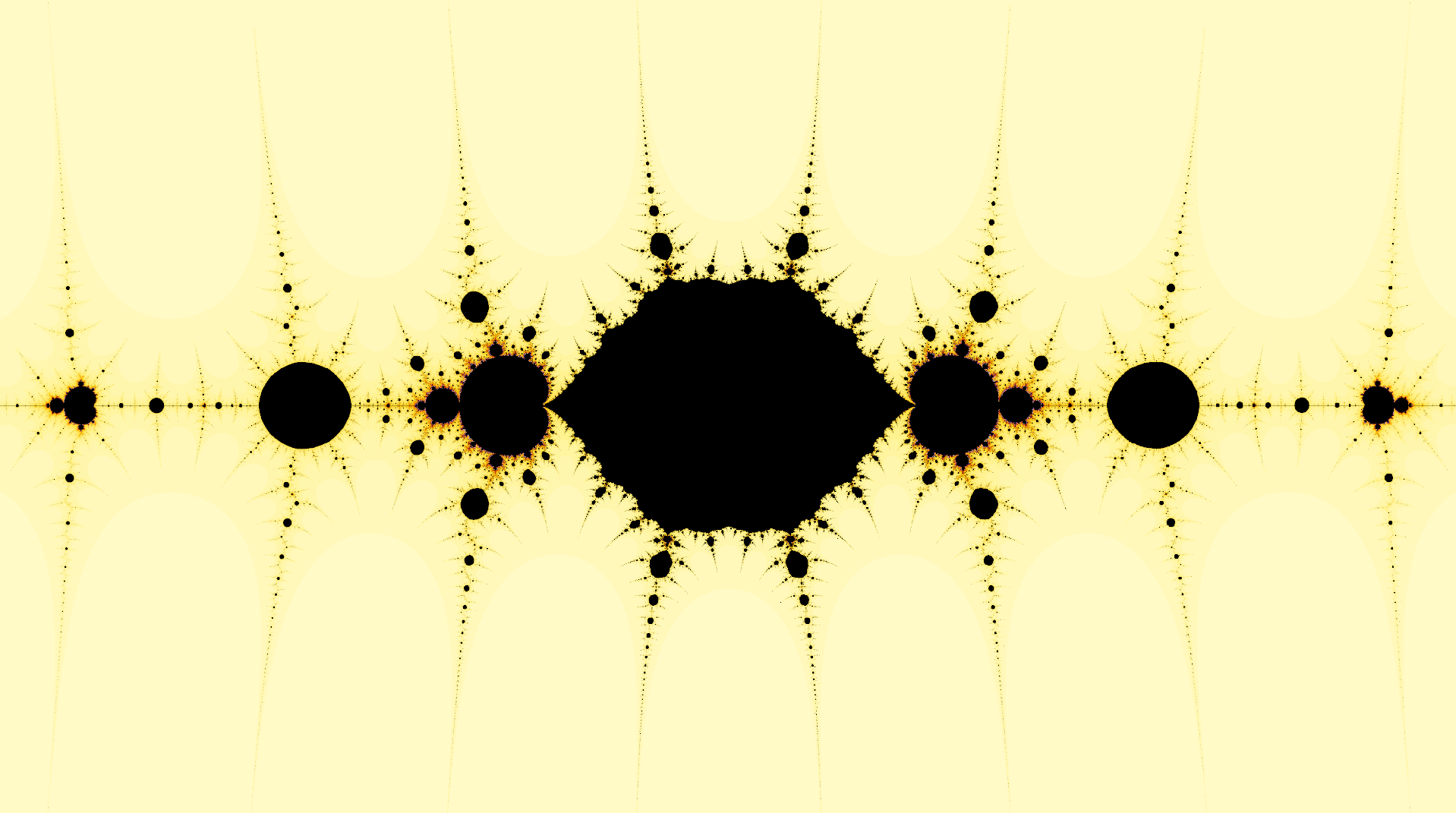}
    \caption{Parameter space of $f_v$, black regions are hyperbolic components.}
\end{figure}

It is easy to see that $f_v$ has critical points $k\pi$, $k \in \mathbb{Z}$, two critical values $0$ and $-2v$, and no finite asymptotic values. Moreover, $0$ is as superattracting fixed point of $f_v$. Denote by $B_v$ the Fatou component containing $0$. 

The hyperbolic maps in $\mathcal{S}_1$ can be defined by 
$$\mathcal{H}=\{f_v: -2v \text{ is attracted by an attracting cycle}\}.$$
Every component of $\mathcal{H}$ is called a hyperbolic component. It is easy to see that there is a hyperbolic component $\mathcal{H}_0$, of which $0$ is an isolated boundary point. And $\mathcal{H}_0$ is a hyperbolic component of type A. 

According to the cycle attracting $-2v$ and the location of $-2v$, we have the classification of hyperbolic components, which is similar to the classification in the cubic polynomials (\cite{Mil09}) and quadratic rational functions (\cite{Ree90}):
\begin{itemize}
    \item Type A (Adjecent): $-2v \in B_v$, so it is attracted by $0$;
    \item Type C (Capture): $-2v \notin B_v$, but there exists an integer $m \ge 1$ such that $f^m_v(-2v)$ belongs to $B_v$;
    \item Type D (Disjoint): $f^n_v(-2v) \notin B_v$ for all $n \ge 0$, so $-2v$ is attracted by another cycle $\{z_v, \ldots, f^{p-1}_v(z_v)\}$.
\end{itemize}

It is easy to see that for a hyperbolic component $\mathcal{U}$ of type C, let 
    \begin{equation*}
        m(v)=\min\{n \ge 1: f^n_v(-2v) \in B_v\}.
    \end{equation*}
Then $m(v)$ is constant on $\mathcal{U}$. Thus we can define $\mathcal{H}_m$ the union of hyperbolic components of type C, in which $m=\min\{j \ge 1: f^j_v(-2v) \in B_v\}$. 

For polynomials, the boundedness of hyperbolic components in the connected locus is obvious. But the result is non-trivial for other families of analytic functions (\cite{Eps00,NP22}). We proved that the hyperbolic components in $\mathcal{S}_1$ are all bounded, which are always unbounded in the exponential family.

\begin{theorem}\label{thm: boundedness of hyperbolic components}
    Every hyperbolic component is a bounded domain in $\mathbb{C}$. The hyperbolic component of type A is unique, which is homeomorphic to the punctured unit disk. Hyperbolic components of type C and D are all homeomorphic to the unit disk. 
\end{theorem}

As stated in the beginning, the local connectivity of the boundaries of hyperbolic components is related to Fatou's conjecture. The local connectivity of the boundaries of hyperbolic components were obtained for many families of polynomials and rational functions (for example, see \cite{Fau92,Roe07,QRWY15,Wan21}). 

In the dynamics of polynomials, a powerful method to study the local connectivity is the Branner-Hubbard-Yoccoz puzzle. For a cosine function $f_{a,b}$, the puzzles can be constructed in a similar way (\cite{QW25}) if the orbits of critical values are bounded. In the family $\{f_v(z)=v(\cos z -1): v \neq 0\}$, we constructed the para-puzzles by the puzzles in the dynamical plane and proved that the boundaries of hyperbolic components in $\mathcal{S}_1$ are all Jordan curves.

\begin{theorem}\label{thm: boundaries of hyperbolic components are Jordan curves}
    $\mathcal{H}_0 \cup \{0\}$ is a Jordan domain. Hyperbolic components of type C and D are Jordan domains. 
\end{theorem}

Theorme \ref{thm: boundaries of hyperbolic components are Jordan curves} implies that all the hyperbolic components except $\mathcal{H}_0$ are Jordan domains. Moreover, the hyperbolic components of type C are of higher regularity.
\begin{theorem}\label{thm: hyperbolic components of type C are quasidisks}
    Hyperbolic components of type C are quasidisks. 
\end{theorem}

This article is structured as follows. In \S \ref{sec: basic properties}, we introduce some properties of the immediate basin $B_v$ of the superattracting fixed point $0$ of $f_v$. The diameter of $B_v$ tends to $0$ as $v$ tends to $\infty$, and the hyperbolic components are simply connected except $\mathcal{H}_0$. In \S \ref{sec: parameterization of hyperbolic components}, we parameterize the hyperbolic components and prove the boundedness of hyperbolic components of type A and C. In \S \ref{sec: Parameter rays}, we review parameter rays and the landing theorem for the rays, which serves as an analog to polynomials. 

Theorem \ref{thm: boundaries of hyperbolic components are Jordan curves} is proved in \S \ref{sec: para-puzzles} and \S \ref{sec: local connectivity of the boundaries of hyperbolic components}. In \S \ref{sec: para-puzzles}, we construct para-puzzles and prove that there is a sequence of non-degenerated annuli surrounding non-renormalizable parameters. In \ref{sec: local connectivity of the boundaries of hyperbolic components}, we treat with renormalizable parameters on the boundaries of hyperbolic components. We show that these parameters are always in a copy of the Mandelbrot set (see Definition \ref{def: copy of Mandelbrot set}). Moreover, we show the local connectivity of boundaries of hyperbolic components at these parameters. Meanwhile, this also induces the boundedness of hyperbolic components of type D and finishes the proof of Theorem \ref{thm: boundedness of hyperbolic components}. Finally, in \S \ref{sec: hyperbolic components of type C are quasidisks}, we prove Theorem \ref{thm: hyperbolic components of type C are quasidisks} by constructing a quasiconformal homeomorphism via a holomorphic motion.

\section{Basic Properties}\label{sec: basic properties}
In this section, we introduce some basic properties of the immediate basin $B_v$ of $0$ and the hyperbolic components. 

\begin{lemma}\label{lm: position of critical points of cosine maps}
    Let $v \in \mathbb{C}^*$. If $f_v$ belongs to a hyperbolic component of type A, then $B_v$ is the unique component of $F(f_v)$, and $k \pi \in B_v$, $k \in \mathbb{Z}$. Otherwise, every Fatou component of $f_v$ contains at most one critical point.
\end{lemma}
\begin{proof}
    Suppose that $f_v$belongs to a hyperbolic component of type A, then all Fatou components of $f_v$ are mapped to $B_v$ eventually. Since $f_v$ is of type A, $-2v \in B_v$. Thus there exists a curve $\gamma \subset B_v$ connecting $0$ and $-2v$. Its preimage $f^{-1}_v(\gamma)$ is a curve, passing through all $k\pi$, $k \in \mathbb{Z}$. And $f^{-1}_v(\gamma)$ is entirely contained in a Fatou component $U$. Since $0 \in f^{-1}_v(\gamma)$, $f^{-1}_v(\gamma) \subset B_v$ and $U=B_v$. Therefore, $f^{-1}_v(B_v)=B_v$. By induction, $B_v$ is the unique component of $F(f_v)$. 

    Now suppose that $v \in \mathbb{C}^*$ does not belong to any hyperbolic components of type A. Suppose by contradiction that $k_1 \pi$, $k_2 \pi$ lie in the same Fatou component. 
    
    \textbf{Case 1: } At least one of $|k_1|, |k_2|$ is even. Since $f_v$ is of $2\pi$-periodic, there exists a critical point $c$ belongs to $B_v$, i.e. $c$ and $0$ are in the same Fatou component. Since $f_v$ is an even function, $-c$ lies in the same Fatou component as $0$. If $c=(2k+1)\pi$ for some $k \in \mathbb{Z}$, then $-2v \in B_v$, contradicting that $f_v$ is not a hyperbolic parameter of type A. Thus $c=2k\pi$. Let $\gamma \subset B_v$ be a Jordan arc connecting $c,-c$, which is sufficiently smooth. We may assume that $(2l+1)\pi \notin \gamma$ for any $l \in \mathbb{Z}$, otherwise $-2v \in B_v$. Then by the Maximum Modulus Principle, the disk surrounded by $\gamma \cup (-\gamma)$ is contained in $B_v$. But $(2l+1)\pi$, $|l| <|k|$ are all contained in this disk. Thus $-2v \in B_v$, again a contradiction.

    \textbf{Case 2: } Both $|k_1|, |k_2|$ are odd. There exists $k \ge 1$ such that $(2k+1)\pi$ lies in the same Fatou component $U$ as $\pi$. We will show that $\pm \pi \in U$, so that $0 \in U$ by the Maximum Modulus Principle. Furthermore, $-2v \in B_v$, which leads a contradiction. If $k=1$, then it follows from the periodicity of $f_v$. Now suppose that $k \ge 2$. We may assume that $(2k-1)\pi \notin U$. Let $\gamma:[0,1] \to U$ be a Jordan arc connecting $(2k+1)\pi$ and $\pi$. Then $\gamma-2\pi$ is a Jordan arc connecting $(2k-1)\pi$ and $-\pi$, and $\gamma \cap (\gamma-2\pi) =\varnothing$. Let $\pi= \gamma(t_0), \gamma(t_1), \ldots, \gamma(t_n)=(2k+1) \pi$ be the intersection points of $\gamma$ and the image axis, where $|\gamma(t_j)|<|\gamma(t_{j+1})|$, $0 \le j \le n-1$. Let $0 \le j_0 \le n-1$ be such that $\mathrm{Re} \; \gamma(t_{j_0})<(2k-1)\pi$, and $\mathrm{Re} \; \gamma(t_{j_0+1})>(2k-1)\pi$. If $|\gamma(t_{j_0+1})-\gamma(t_{j_0})|>2\pi$, let $\epsilon>0$ be sufficiently small. Without loss of generality, we assume that $\mathrm{Im} \; \gamma(t)>0$, $t \in (t_{j_0}+\epsilon, t_{j_0+1}-\epsilon)$. Let $D$ be the domain surrounded by $[\gamma(t_{j_0}), \gamma(t_{j_0+1})] \cup \gamma([t_{j_0}, t_{j_0+1}])$. Then $\gamma([t_{j_0}+\epsilon, t_{j_0+1}-\epsilon])$ connects a point $\gamma(t_{j_0+1}-\epsilon)$ in $D$ and a point $\gamma(t_{j_0}+\epsilon)$ outside $D$, and $\gamma([t_{j_0}+\epsilon, t_{j_0+1}-\epsilon])-2\pi \cap [\gamma(t_{j_0}), \gamma(t_{j_0+1})] =\varnothing$. Therefore, $\gamma \cap (\gamma-2\pi) \neq \varnothing$, contradicting that $\gamma$ is disjoint from $\gamma-2\pi$. Similarly, we have $\gamma \cap (\gamma -2\pi) \neq \varnothing$ if $|\gamma(t_{j_0+1})-\gamma(t_{j_0})|<2\pi$. 

    Therefore, for a parameter $v$ not in any hyperbolic components of type A, every Fatou component of $f_v$ contains at most one critical point. 
\end{proof}

In \cite{QW25}, the authors proved the following two results: 
\begin{corollary}\label{cor: boundedness of B_v}
    Let $v \in \mathbb{C}^*$. If $f_v$ belongs to a hyperbolic component of type A, then $B_v$ is unbounded, and $\partial B_v$ is not locally connected anywhere. Otherwise, $B_v$ is bounded. 
\end{corollary}

\begin{corollary}\label{cor: local connectivity of B_v}
    Let $v \in \mathbb{C}^*$, which is not in any hyperbolic components of type A. If $P(f_v)$ is bounded or the orbit of $-2v$ escapes to $\infty$, then $B_v$ is a Jordan domain.  
\end{corollary}

We can estimate the diameter of $B_v$.
\begin{lemma}\label{lm: diameter of B_v tends to 0}
    $\mathrm{diam} \; B_v \to 0$ as $v \to \infty$. 
\end{lemma}
\begin{proof}
    Let $h_1(z)=\cos z-1$, $h_2(z)=z/v$. Let $\epsilon>0$ sufficiently small and $Q=\{z \in \mathbb{C}: |\mathrm{Re} \; z| <\epsilon, |\mathrm{Im} \; z| <\epsilon\}$. Then $h_1(Q)$ is a bounded domain surrounded by an ellipse and a hyperbola, containing a focus $0$, and $h_2(Q)$ is a square obtained by a rotation and a rescaling of $Q$. It is clear that $h_2(Q) \subset D(0,\sqrt{2}\epsilon/|v|)$. $d(0,\partial h_1(Q))=1-\cos \epsilon>\frac{\epsilon^2}{2\sqrt{2}}$. Thus for $|v|$ sufficiently large, for $\epsilon=4/|v|$, $f_v: Q \to f_v(Q)$ is a renormalization (see \cite{McM94} for the definition of renormalization). This implies that $B_v \subset Q$, so that $\mathrm{diam} \; B_v \le \mathrm{diam } \; Q \le 8\sqrt{2}/|v|$. 
\end{proof}

Conversely, the parameters in a punctured disk of center $0$ and sufficiently small radius $r>0$ are hyperbolic and of type A.
\begin{lemma}\label{lm: H_0 contains a small disk}
    For $0<|v|<1/5$, $-2v \in B_v$, i.e. $f_v$ is hyperbolic and of type A.
\end{lemma}
\begin{proof}
    Consider a rectangle $R=\{z=x+iy: |x|<\pi, |y|<\log 1/|v|\}$. Then $f_v(R)$ is an ellipse with two foci $0$, $-2v$, and the length of the major axis is $1+|v|^2$. Then for $0<|v|<1$, $f_v(R) \subset \mathbb{D} \subset R$. Thus $R \subset B_v$. Moreover, $|-2v|<2/5$ so that $-2v \in R$. This implies that $-2v \in B_v$, i.e. $v$ is a hyperbolic parameter of type A.
\end{proof}

By Lemma \ref{lm: H_0 contains a small disk}, there is a hyperbolic component of type A which contains a sufficiently small punctured disk of center $0$. Denote this component by $\mathcal{H}_0$. Now we show that the hyperbolic components are simply connected, except for $\mathcal{H}_0$, which is doubly connected. 

\begin{lemma}\label{lm: hyperbolic components are simply connected}
    Let $\mathcal{U}$ be a hyperbolic component. If $\mathcal{U} \neq \mathcal{H}_0$, then $\mathcal{U}$ is simply connected. And $\mathcal{H}_0$ is doubly connected. 
\end{lemma}
\begin{proof}
    If $\mathcal{U}$ is a hyperbolic component of type D, then the proof is the same as the family of quadratic polynomials (see \cite{CG93}, Chapter VIII, Theorem 2.1). 

    Suppose that $\mathcal{U}$ is a hyperbolic component of type C. Let $\gamma \subset \mathcal{U}$ be an arbitrary bounded Jordan curve, and $\Omega$ be the domain surrounded by $\gamma$. It will be sufficient to show that $\Omega \subset \mathcal{U}$. Indeed, by the compactness of $\gamma$, for sufficiently small $\epsilon >0$, there exists $N \ge m$ such that $f^N_v(-2v) \in D(0,\epsilon)$ for all $v \in \gamma$. The function $f^N_v(-2v)$ is analytic on $\mathbb{C}^*$. By the Maximum Modulus Principle, $f^N_v(-2v) \in D(0,\epsilon)$ for all $v \in \overline{\Omega}$. Thus $\Omega$ cannot intersect with the complement of $\mathcal{U}$. In other words, $\Omega \subset \mathcal{U}$, i.e. $\mathcal{U}$ is simply connected. 

    Similarly, $\mathcal{H}_0 \cup \{0\}$ is simply connected so that $\mathcal{H}_0$ is doubly connected. 
\end{proof}

\section{Parameterization of hyperbolic components}\label{sec: parameterization of hyperbolic components}
As a result of Lemma \ref{lm: hyperbolic components are simply connected}, the hyperbolic components can be mapped to either the unit disk $\mathbb{D}$ or the punctured unit disk $\mathbb{D}^*$ conformally. The mappings can be chosen according to the dynamics of $f_v$. Moreover, we will show that all the hyperbolic components are bounded domains in $\mathbb{C}$, comparing with the exponential family, whose hyperbolic components are all unbounded (see \cite{RS09}). 

\subsection{Hyperbolic components of type A}\label{ss: hyperbolic components of type A}
Let $\mathcal{U}$ be a hyperbolic component of type A. Recall that $B_v$ is the immediate basin of $0$, $\phi_v: U_v \to D(0,r)$ is the B\"ottcher map defined on a neighborhood $U_v$ of $0$, satisfying $\phi_v \circ f_v(z)=\phi_v(z)^2$, $\phi_v(0)=0$, and the free critical value $-2v \in \partial U_v$. Since $f_v(z)=f_v(-z)$, $U_v$ is symmetric about $0$, and $\phi_v(-z)=-\phi_v(z)$. Define
\begin{equation}\label{eq: Phi_0}
    \begin{split}
        \Phi_0: \mathcal{U} & \to \mathbb{D} \\
        v & \mapsto \phi_v(-2v)
    \end{split}
\end{equation}

\begin{lemma}\label{lm: hyperbolic component of type A maps to the punctured unit disk}
    $\Phi_0(\mathcal{U}) \subset \mathbb{D}^*$. 
\end{lemma}
\begin{proof}
    If there exists a $v \in \mathcal{U}$ such that $\Phi_0(v)=0$, then $-2v$ is the fixed point $0$ of $f_v$. Thus $v=0$, a contradiction.
\end{proof}

The boundedness of $\mathcal{U}$ follows from Corollary \ref{cor: boundedness of B_v} and Lemma \ref{lm: diameter of B_v tends to 0}:
\begin{lemma}\label{lm: boundedness of H_0}
    Let $\mathcal{U}$ be a hyperbolic component of type A, then $\mathcal{U}$ is bounded.
\end{lemma}
\begin{proof}
    By Lemma \ref{lm: diameter of B_v tends to 0}, there exists $R>0$ such that $B_v$ is bounded for all $|v|>R$. By Corollary \ref{cor: boundedness of B_v}, $B_v$ is unbounded for all $v \in \mathcal{U}$. Thus $\mathcal{U} \subset D(0,R)$. 
\end{proof}

In fact, the map $\Phi_0$ defined by \eqref{eq: Phi_0} gives an analytic covering from $\mathcal{U}$ to $\mathbb{D}^*$. Moreover, the hyperbolic component of type A is unique. 
\begin{proposition}\label{prop: Phi_0 is a covering map of degree 2}
    $\mathcal{H}_0$ is the unique hyperbolic component of type A, and $\Phi_0: \mathcal{H}_0 \to \mathbb{D}^*$ given by \eqref{eq: Phi_0} is an analytic covering of degree $2$. 
\end{proposition}
\begin{proof}
    Let $\mathcal{U}$ be a hyperbolic component of type A. We show that $\Phi_0: \mathcal{U} \to \mathbb{D}^*$ is an analytic covering of finite degree. 
    
    First of all, since $\phi_v(z)$ is analytic with respect to both $v$ and $z$, $\Phi_0$ is an analytic function. 

    Then we show that $\Phi_0: \mathcal{U} \to \mathbb{D}^*$ is a proper map. It will be sufficient to show that $\Phi_0(v_n) \to \partial \mathbb{D}$ as $v_n \to v_0 \in \partial \mathcal{U}$. Let $G_n(z)= \log |\phi_{v_n}(z)|$ be the Green function on $B_{v_n}$. For $z \notin \bigcup_{n \ge 0} f^{-n}(B_{v_n})$, define $G_{v_n}(z)=0$. Then $G_{v_n}$ can be extended to $\mathbb{C}^*$ continuously. Similarly, $G_{v_0}$ is continuous on $\mathbb{C}^*$. Since $v_0 \in \partial \mathcal{U}$, $-2v_0 \notin \bigcup_{n \ge 0} f^{-n}(B_{v_0})$, i.e. $G_{v_0}(-2v_0)=0$. Therefore, $|\Phi_0(v_n)|=|\phi_{v_n}(-2v_n)|=e^{G_n(-2v_n)} \to e^{G_0(-2v_0)}=1$ as $n \to \infty$. So $\Phi_0$ is a proper map. Furthermore, it is an analytic branched covering map. 
    
    Since the B\"ottcher map $\phi_v$ is conformal, $\Phi'_0(v) \neq 0$ if $v \neq 0$. Thus for $\mathcal{U} \neq \mathcal{H}_0$, $\Phi_0$ has no branch points. Thsi implies that $\Phi_0: \mathcal{U} \to \mathbb{D}^*$ is an analytic covering map of finite degree. However, such a covering map does not exist, since $\mathcal{U}$ is simply connected by Lemma \ref{lm: hyperbolic components are simply connected}. In other words, $\mathcal{H}_0$ is the unique hyperbolic component of type A.

    Finally we show that $\Phi_0: \mathcal{H}_0 \to \mathbb{D}^*$ is an analytic covering map of degree $2$. By Riemann's Removable Singularities Theorem, $\Phi_0$ can be extended to an analytic function on $\mathcal{H}_0 \cup \{0\}$. So $\Phi_0: \mathcal{H}_0 \cup \{0\} \to \mathbb{D}$ is an analytic branched covering of finite degree. Since $0$ is the unique zero of $\Phi_0$, $0$ is the unique branch point of $\Phi_0$. It will be sufficient to show that $0$ is the zero of $\Phi_0$ of degree $2$. Since $f_v(z)=v(\cos z-1)=\frac{vz^2}{2}(1+O(z^2))$, by induction on $n$, we have 
    \begin{equation*}
        f^n_v(z)=(\frac{v}{2})^{2^n-1}z^{2^n}(1+O(z^2)).
    \end{equation*}
    By the definition of $\phi_v(z)$,  
    \begin{equation*}
        \Phi_0(v)=\phi_v(-2v)=\lim_{n \to \infty} (f^n_v(-2v))^{1/2^n}=v^2(1+O(v^2)).
    \end{equation*}
    Thus $0$ is the zero of $\Phi_0$ of degree $2$. So $\Phi_0: \mathcal{H}_0 \to \mathbb{D}^*$ is an analytic covering map of degree $2$.  
\end{proof}

\subsection{Hyperbolic components of type C}\label{ss: hyperbolic components of type C}
Recall that for $m \ge 1$, $\mathcal{H}_m=\{v \in \mathbb{C}^*: f^m_v(-2v) \in B_v, f^{m-1}_v(-2v) \notin B_v\}$, and $\mathcal{H}_0$ is the unique hyperbolic component such that $-2v \in B_v$. Thus one can define $\hat{\mathcal{H}}=\bigcup_{m \ge 0} \mathcal{H}_m$.

Let $\mathcal{U}$ be a component of $\mathcal{H}_m$, $m \ge 1$. By Lemma \ref{lm: position of critical points of cosine maps}, $B_v$ contains only one critical point $0$. Thus $\phi_v$ can be extended to a conformal map $\phi_v:B_v \to \mathbb{D}$. Define 
\begin{equation}\label{eq: Phi_U}
    \begin{split}
        \Phi_{\mathcal{U}}: \mathcal{U} & \to \mathbb{D} \\
        v & \mapsto \phi_v(f^m_v(-2v))
    \end{split}.
\end{equation}

First we show that $\mathcal{U}$ is bounded. $f^m_v(-2v) \in B_v$, $v \in \mathcal{U}$. There exists an integer $k_v \in \mathbb{Z} \setminus \{0\}$, such that $f^{m-1}_v(-2v) \in B_v +2k_v\pi$. 
\begin{lemma}\label{lm: k_v is constant}
    The function $v \mapsto k_v$ is a constant on $\mathcal{U}$. In other words, there exists $k \in \mathbb{Z}$, $k \neq 0$, such that $f^{m-1}_v(-2v) \in B_v+2k\pi$ for all $v \in \mathcal{U}$.
\end{lemma}
\begin{proof}
    It will be sufficient to show that $k_v$ is a constant locally. Let $v \in \mathcal{U}$, $z(v)=f^{m-1}_v(-2v)$. $z(v)$ is an analytic function of $v$. Choose a neoghborhood $\mathcal{N} \subset \mathcal{U}$ of $v$, such that $\partial B_v$ is continuous on $\mathcal{N}$, with respect to the Hausdorff topology. Hence, if $z(v) \in B_v+2k\pi$, then $z(v') \in B_{v'}+2k\pi$ for all $v' \in \mathcal{N}$. 
\end{proof}

\begin{lemma}\label{lm: boundness of components of type C}
    Let $\mathcal{U}$ be a hyperbolic component of type C. Then $\mathcal{U}$ is bounded.
\end{lemma}
\begin{proof}
    If $\mathcal{U}$ is unbounded, then exists $v_n \in \mathcal{U}$, $v_n \to \infty$. By Lemma \ref{lm: diameter of B_v tends to 0}, $\mathrm{diam} \; B_{v_n} \to 0$ as $n \to \infty$. Thus $f^m_{v_n}(-2v_n)=v_n(\cos f^{m-1}_{v_n}(-2v_n)-1)$ is bounded. So there exists $k_n \in \mathbb{Z}$, such that $|f^{m-1}_{v_n}(-2v_n)-2k_n \pi | \to 0$. By Lemma \ref{lm: k_v is constant}, $k_n$ is a constant $k$. So $f^{m-1}_{v_n}(-2v_n) \to 2k \pi$. If $m=1$, then $-2v_n \to 2k\pi$, a contradiction. If $m \ge 2$, then $f^{m-1}_{v_n}(-2v_n)=v_n (\cos f^{m-2}_{v_n}-1) \to 2k\pi$ as $n \to \infty$. So $f^{m-2}_{v_n}(-2v_n) \to 2k'\pi$ for some $k' \in \mathbb{Z}$. By induction on $m$, we have $-2v_n \to 2k'' \pi$ as $n \to \infty$, contradicting that $v_n \to \infty$.
\end{proof}

\begin{proposition}\label{prop: Phi_U is a conformal isomorphism}
    The map $\Phi_{\mathcal{U}}:\mathcal{U} \to \mathbb{D}$ defined by \eqref{eq: Phi_U} is a conformal isomorphism. 
\end{proposition}
\begin{proof}
    First of all, $\Phi_{\mathcal{U}}$ is an analytic function. By the boundedness of $\mathcal{U}$, $\Phi_{\mathcal{U}}: \mathcal{U} \to \mathbb{D}$ is a branched covering map similar to the proof of Proposition \ref{prop: Phi_0 is a covering map of degree 2}. 
    
    We claim that $\Phi_{\mathcal{U}}$ has no branch points. Let $v_0 \in \mathcal{U}$, $\Phi_{\mathcal{U}}(v_0)=z_0 \in \mathbb{D}$. Let $\epsilon >0$ be sufficiently small, $z \in D(z_0, \epsilon)$. As in the proof of Lemma 2.8 in \cite{Roe07}, we can conduct a quasiconformal surgery to obtain an entire function $h_z$, satisfying   
    \begin{enumerate}
        \item $h_z(w)=h_z(w+2\pi)$;
        \item $\rho(h_z)=1$, here $\rho$ is the order of a transcendental entirre function;
        \item $h_z$ has no asymptotic values;
        \item the zeros of $h_z$ are $2k\pi$, $k \in \mathbb{Z}$, and they are all of degree $2$.
    \end{enumerate}
    By Hadamard's Factorization Theorem, we have  
    \begin{equation*}
        \begin{split}
            h_z(w) & = w^2 e^{aw+b} \prod_{k=1}^{\infty} \left(1-\left(\frac{z}{2k\pi}\right)^2\right)^2 \\
            & =C e^{aw}(\cos w-1)
        \end{split} 
    \end{equation*}
    where $C \in \mathbb{C}$, $a \in \mathbb{C}$. Since $h_z(w+2\pi)=h_z(w)$, $a \in \{0,1,-1\}$. If $a \neq 0$, then $h_z$ has an asymptotic value $C/2$, a contradiction. Hence $a=0$, $h_z(w)=v(z) (\cos w-1)$, i.e. $h_z=f_{v(z)}$. 
    
    Again, similar to the cubic polynomials (\cite{Roe07}), one can easily verify that $\Phi_{\mathcal{U}}(v(z))=z$. This implies that $\Phi_{\mathcal{U}}:\mathcal{U} \to \mathbb{D}$ is a local homeomorphism, hence has no branch points. Since $\mathcal{U}$ and $\mathbb{D}$ are both simply connected, $\Phi_{\mathcal{U}}: \mathcal{U} \to \mathbb{D}$ is a homeomorphism, hence a conformal isomorphism. 
\end{proof}

By Proposition \ref{prop: Phi_U is a conformal isomorphism}, for every component $\mathcal{U}$ of $\mathcal{H}_m$, there is a unique $v_{\mathcal{U}} \in \mathcal{U}$ such that $\Phi_{\mathcal{U}}(v_{\mathcal{U}})=0$, so that $f^m_{v_{\mathcal{U}}}(-2v_{\mathcal{U}})=0$. We call $v_{\mathcal{U}}$ the center of $\mathcal{U}$.

\subsection{Hyperbolic components of type D}\label{ss: hyperbolic components of type D}
Let $\mathcal{D}$ be a hyperbolic component of type D. Every $v \in \mathcal{D}$, the free critical value $-2v$ of $f_v$ is attracted by an attracting cycle $\{z_v, f_v(z_v), \ldots, f^{p-1}_v(z_v)\}$. By the Implicit Function Theorem, $z_v$ is an analytic function with respect to $v$, whose period $p$ is a constant. Let $\lambda(v)=(f^p_v)'(z_v)$. Then $\lambda: \mathcal{D} \to \mathbb{D}$ is an analytic function. 

\begin{lemma}\label{lm: boundedness of hyperbolic components of type D}
    $\mathcal{D}$ is a bounded domain in $\mathbb{C}$. 
\end{lemma}

Lemma \ref{lm: boundedness of hyperbolic components of type D} will be proved in \S \ref{ss: renormalizable parameters on the boundaries of hyperbolic components}. Now Theorem \ref{thm: boundedness of hyperbolic components} follows from Lemma \ref{lm: hyperbolic components are simply connected}, \ref{lm: boundedness of H_0}, \ref{lm: boundness of components of type C} and \ref{lm: boundedness of hyperbolic components of type D}. 

Using Lemma \ref{lm: boundedness of hyperbolic components of type D}, we can show that $\lambda: \mathcal{D} \to \mathbb{D}$ is a conformal isomorphism, and can be extended to $\overline{\mathcal{D}} \to \overline{\mathbb{D}}$ homeomorphically, so that $\mathcal{D}$ is a Jordan domain. 

\begin{proposition}\label{prop: parameterization of hyperbolic components of type D}
    $\lambda: \mathcal{D} \to \mathbb{D}$ is a conformal isomorphism, and can be extended to $\lambda: \overline{\mathcal{D}} \to \overline{\mathbb{D}}$ homeomorphically. In particular, $\mathcal{D}$ is a Jordan domain.
\end{proposition}
\begin{proof}
    By Lemma \ref{lm: boundedness of hyperbolic components of type D}, $\mathcal{D}$ is bounded. As in the proof of Lemma \ref{prop: Phi_0 is a covering map of degree 2}, $\lambda: \mathcal{D} \to \mathbb{D}$ is a proper map. This implies that $\lambda$ is a branched covering map. Similar to the family of quadratic polynomials (\cite{CG93}, Chapter VIII, Theorem 2.1), $\lambda$ is a local homeomorphism, thus it has no branch points. Since $\mathcal{D}$ and $\mathbb{D}$ are both simply connected, $\lambda$ is a conformal isomorphism. 

    Again, similar to \cite{CG93}, Chapter VIII, Theorem 2.1, $\lambda$ can be extended to the boundary of $\mathcal{D}$ homeomorphically. In other words, $\mathcal{D}$ is a Jordan domain. 
\end{proof}

\section{Parameter rays}\label{sec: Parameter rays}
From now on, we will prove Theorem \ref{thm: boundaries of hyperbolic components are Jordan curves}. In this section, we review the parameter rays for the cosine family. 

\subsection{Dynamic rays}\label{ss: dynamic rays}
Recall that for the cosine function $f_v(z)=v(\cos z-1)$, the escaping set of $f_v$ is defined as 
$$I(f_v)=\{z \in \mathbb{C}: f^n_v(z) \to \infty \text{ as } n \to \infty\}.$$

As shown in \cite{RS08}, $I(f_v)$ consists of uncountably many curves, called dynamic rays, which serves as an analog to the external rays of polynomials. We introduce the dynamic rays briefly. 

Let $\Gamma=[-2v,0] \cup \mathbb{R}_{\ge 0}$. Then $f^{-1}_v(\mathbb{C} \setminus \Gamma)$ are countably many half strips $P_{j,k}$, $j=0,1$, $k \in \mathbb{Z}$. Here $P_{0,k}$ are contained in the upper half plane $\{z \in \mathbb{C}: \mathrm{Im} \; z>0\}$, while $P_{1,k}$ are contained in the lower half plane, and $P_{j,k+1}=P_{j,k}+2\pi$. The strips are labeled such that the segment $[2k\pi, 2(k+1)\pi]$ is contained in the boundary of both $P_{0,k}$ and $P_{1,k}$. Then $f_v: P_{j,k} \to \mathbb{C} \setminus \Gamma$ is a conformal isomorphism, and its inverse is denoted by $L_{j,k}$. 

Let $\Sigma=(\{0,1\} \times \mathbb{Z})^{\mathbb{N}}$, $\sigma: \Sigma \to \Sigma$ be the shift. A sequence $\underline{s}=((j_n, k_n))_{n \ge 0}$ is called exponentially bounded, if there exists $x>0$, such that $|k_n| \le F^{n}(x)$ for all $n \ge 0$. Let $F(x)=e^x-1$. The sequence of functions 
\begin{equation}\label{eq: g^v_n,s}
    g^v_{\underline{s},n}(t)=L_{j_0, k_0} \circ \cdots \circ L_{j_n, k_n} \circ F^n
\end{equation}
converges to a function $g^v_{\underline{s}}$ uniformly for sufficiently large $t>0$. $g^v_{\underline{s}}$ is defined on an interval $(t_{\underline{s}}, +\infty)$, called a dynamic ray with the external address $\underline{s}$. The definition of $t_{\underline{s}}$ can be found in \cite{RS08}, Definition 2.6. From \eqref{eq: g^v_n,s}, it is easy to see that $f_v(g^v_{\underline{s}}(t))=g^v_{\sigma(\underline{s})}(F(t))$ as long as $g^v_{\underline{s}}(t)$ and $g^v_{\sigma(\underline{s})}(F(t))$ are both defined. The ray $g^v_{\underline{s}}$ is said to land at some point $z \in \mathbb{C}$, if 
$$\lim_{t \to t_{\underline{s}}+} g^v_{\underline{s}}(t)=z. $$
The dynamic ray $g^v_{\underline{s}}$ is contained in the escaping set $I(f_v)$. On the contrary, every $z \in I(f_v)$ is either on some dynamic ray, or the landing point of some ray (\cite{RS08}, Theorem 6.4). 

We end Section \ref{ss: dynamic rays} with a landing theorem of the dynamic rays (see \cite{BR20}, Theorem 1.4).  
\begin{theorem}\label{thm: landing theorem}
    Let $f_v$ be a cosine function, $P(f_v)$ be its post-critical set. Suppose that $P(f_v)$ is bounded. Then 
    \begin{enumerate}
        \item every dynamic ray $g^v_{\underline{s}}$ with $\underline{s}$ periodic under $\sigma$ lands at a repelling or parabolic periodic point;
        \item for every repelling or parabolic periodic point, there is at least one and at most finitely many dynamic rays landing at it.
    \end{enumerate}
\end{theorem}

\subsection{Parameter rays}\label{ss: parameter rays}
The escaping parameter set is defined as 
$$\mathcal{I}=\{v \in \mathbb{C}^*: f^n_v(-2v) \to \infty \text{ as } n \to \infty\}. $$
It has been shown in \cite{Tian11} that $\mathcal{I}$ also has the structure of rays, which is the counterpart of the parameter external rays of the family of polynomials. In \cite{Tian11}, Tian considered the family $\{E_{\kappa}(z)=e^{\kappa} (e^z+e^{-z}): \kappa \in \mathbb{C}\}$. The method in \cite{Tian11} still works here. So we present the results and omit the proof. 
\begin{proposition}\label{prop: parameter rays}
    Let $\Sigma$ be the symbol space defined in Section \ref{ss: dynamic rays}, $\underline{s} \in \Sigma$. There is a unique injective curve $G_{\underline{s}}: (t_{\underline{s}}, +\infty) \to \mathcal{I}$ such that $v=G_{\underline{s}}(t)$ if and only if the critical value $-2v=g^v_{\underline{s}}(t)$ in the dynamical plane of $f_v$. 
\end{proposition}

A parameter ray $G_{\underline{s}}$ is said to land at $v_0 \in \mathbb{C}$, if 
$$\lim_{t \to t_{\underline{s}}+} G_{\underline{s}}(t)=v_0.$$

For polynomials, the landing theorem of parameter external rays with rational external angles is well-known (for example, see \cite{DH84, CG93}). Schleicher (\cite{Sch99}) established the landing theroem for the exponential family. For the cosine family, the landing theorem of parameter rays can be obtained similarly.
\begin{proposition}\label{prop: landing theorem for parameter rays}
    Let $\underline{s} \in \Sigma$ be eventually periodic under the shift $\sigma$. Then $G_{\underline{s}}$ lands at a parameter which is either parabolic or strictly eventually preperiodic. More precisely, we have the following results.
    \begin{enumerate}
        \item if $\underline{s}$ is periodic, then $G_{\underline{s}}$ lands at a parameter $v$ such that $f_v$ has a parabolic periodic point;
        \item if $\underline{s}$ is strictly preperiodic, then $G_{\underline{s}}$ lands at a parameter such that the critical value $-2v$ of $f_v$ is strictly preperiodic. 
    \end{enumerate}
\end{proposition}

\begin{proposition}\label{prop: converse landing theorem for eventually periodic parameters}
    Suppose that $v$ is a parameter such that $-2v$ is strictly preperiodic. Then $v$ is the landing point of $G_{\underline{s}}$ if and only if $g^v_{\underline{s}}$ lands at $-2v$ in the dynamical plane of $f_v$.  
\end{proposition}

\begin{definition}\label{def: characteristic point}
    Let $v$ be a parabolic parameter. $\{z_0(v), \ldots, z_{p-1}(v)\}$ is the parabolic cycle. The characteristic point of the cycle is the point $z_j(v)$ on the boundary of the Fatou component containing $-2v$.  
\end{definition}

\begin{proposition}\label{prop: converse landing theorem for parabolic parameters}
    Suppose that $v$ is a parabolic parameter. Let $z(v)$ be the characteristic point of the parabolic cycle of $f_v$. Then $G_{\underline{s}}$ lands at $v$ if and only if $g^v_{\underline{s}}$ lands at $z(v)$. 
\end{proposition}

\subsection{Parameter internal rays}\label{ss: parameter internal rays}
Suppose that $v \in \mathbb{C}^*$ such that $-2v \notin B_v$. Then the B\"ottcher mapping extends to a conformal isomorphism $\phi_v: B_v \to \mathbb{D}$, with $\phi_v(0)=0$, $\phi_v(f_v(z))=\phi^2_v(z)$. In this case, we can define internal rays in $B_v$. For $\theta \in \mathbb{R}/\mathbb{Z}$, define $R^v_0(\theta)=\phi^{-1}_v((0,1)e^{2\pi i \theta})$, which is called the internal ray of angle $\theta$. It is easy to see that $f_v(R^v_0(\theta))=R^v_0(2\theta)$.

$R^v_0(\theta)$ is said to land at $z$ if 
$$\lim_{r \to 1-} \phi^{-1}_v(re^{2\pi i \theta}) =z.$$
The following landing theorem is similar to polynomials.
\begin{proposition}\label{prop: landing theorem for internal rays of rational angles}
    Suppose that $v \in \mathbb{C}^*$ such that the internal rays are well-defined, and $\theta \in \mathbb{Q}/\mathbb{Z}$. Then $R^v_0(\theta)$ lands at a preperiodic point $z \in \partial B_v$. Conversely, every preperiodic point $z \in \partial B_v$ is the landing point of at least one and at most finitely many internal rays.
\end{proposition}
 
By the parameterization of hyperbolic components, we can define the parameter internal rays. Consider the hyperbolic component $\mathcal{H}_0$. By Proposition \ref{prop: Phi_0 is a covering map of degree 2}, $\Phi_0: \mathcal{H}_0 \to \mathbb{D}^*$ is a covering map of degree $2$. By the symmetry of $\mathcal{H}_0$, we consider $\mathcal{H}_0 \cap \mathbb{H}$. It is easy to see that $\Phi^{-1}_0((0,1))=(-v_1,0) \cup (0,v_1)$, where $v_1$ is the parameter in Lemma \ref{lm: maps on the real axis}. So define $\mathcal{R}_0(0)=(-v_1,0) \cup (0,v_1)$. For $\theta \in (0,1)$, the parameter internal ray $\mathcal{R}_0(\theta)$ of angle $\theta$ is the component of $\Phi^{-1}_0((0,1) e^{2\pi i \theta})$ in the upper half plane. 

Similarly, for a hyperbolic component $\mathcal{U}$ of type $C$, by Proposition \ref{prop: Phi_U is a conformal isomorphism}, $\Phi_{\mathcal{U}}: \mathcal{U} \to \mathbb{D}$ is a conformal isomorphism. The parameter internal ray $\mathcal{R}_{\mathcal{U}}(\theta)$ of angle $\theta \in \mathbb{R}/\mathbb{Z}$ is defined as $\Phi^{-1}_{\mathcal{U}}((0,1)e^{e\pi i \theta})$. 

A parameter internal ray $\mathcal{R}_0(\theta)$ (or $\mathcal{R}_{\mathcal{U}}(\theta)$) is said to land at $v$ if 
$$\lim_{t \to 1-} \Phi^{-1}_0(te^{2\pi i \theta})=v \quad (\text{or} \lim_{t \to 1-} \Phi^{-1}_{\mathcal{U}}(te^{2\pi i \theta})=v \text{ respectively}).$$

Here, $\Phi^{-1}_0$ is the branch of the inverse of $\Phi_0$ maps $\mathbb{D}^* \setminus (0,1)$ into the upper half plane. 

The following is similar to Proposition \ref{prop: landing theorem for parameter rays}. 
\begin{proposition}\label{prop: landing theorem for parameter internal rays}
    Let $\mathcal{U}$ be a component of $\mathcal{H}_m$, $m \ge 0$, and $\mathcal{R}_0(\theta)$ ($\mathcal{R}_{\mathcal{U}}(\theta)$ resp.) be a parameter internal ray of angle $\theta \in \mathbb{Q}/\mathbb{Z}$. Then the ray lands at some $v_0 \in \mathbb{C}^*$. Moreover, either $f_{v_0}$ has a parabolic cycle or the internal ray $R^{v_0}_0(\theta)$ in the dynamical plane of $f_{v_0}$ lands at $f^m_{v_0}(-2v_0)$, which is eventually repelling.
\end{proposition}

\section{Para-puzzles}\label{sec: para-puzzles}
In this sectioin, we will construct the para-puzzles for the cosine family.

\subsection{Maps on the boundaries of hyperbolic components}\label{ss: maps on the boundaries of hyperbolic components}
First we consider maps on the boundaries of hyperbolic components of type A and C.
\begin{lemma}\label{lm: renormalizable cosine function}
    Suppose that $-2v \notin B_v$, and $\partial B_v$ contains no parabolic periodic points nor critical points $(2k+1) \pi$, $k \in \mathbb{Z}$. Then there are topological disks $U \Subset V$, $m \ge 1$, with $B_v \Subset U$, such that $f^m_v: U \to V$is a polynomial-like mapping of degree $2^m$. 
\end{lemma}
\begin{proof}
    If $\partial B_v$ contains no parabolic periodic points nor critical points $(2k+1)\pi$, $k \in \mathbb{Z}$, then $f_v$ is semi-hyperbolic on $\partial B_v$ (\cite{BM02}). The construction of $U, V$ is similar to \cite{QRWY15}, Lemma 4.2.     
\end{proof}

\begin{lemma}\label{lm: maps on the boundary of H_0}
    Suppose that $v \in \partial \mathcal{H}_0$, $v \neq 0$. Then $\partial B_v$ contains either a parabolic cycle or the free critical value $-2v$.
\end{lemma}
\begin{proof}
    First note that if $\partial B_v$ contains a parabolic periodic point, then it contains the entire cycle of this parabolic periodic point. Suppose that $v \in \partial \mathcal{H}_0$, $v \neq 0$ such that $\partial B_v$ does not contain a parabolic cycle, and $-2v \notin \partial B_v$. By Lemma \ref{lm: renormalizable cosine function}, there are topological disks $U_v \Subset V_v$, $m \ge 1$, such that $B_v \Subset U_v$, $f^m_v: U_v \to V_v$ is a polynomial-like mapping of degree $2^m$,  and $f^m: U_v \to V_v$ has the unique critical point $0$. We may shrink $U_v$ such that 
    $$
    \overline{V_v} \cap \left(\bigcup_{0 \le j \le m} f^{-j}_v(\{(2k+1) \pi: k \in \mathbb{Z}\}) \right) = \varnothing.
    $$
    By the continiuty of $f^{-j}_v$, $0 \le j \le m$, there is a neoghborhood $\mathscr{V}$ of $v$ such that for all $u \in \mathscr{V}$, 
    $$
    \overline{V_v} \cap \left( \bigcup_{0 \le j \le m} f^{-j}_u(\{(2k+1)\pi: k \in \mathbb{Z}\})\right) = \varnothing.
    $$
    Let $U_u$ is the component of $f^{-m}_u(V_v)$ containing $0$. We may assume that $U_u \Subset V_v$, $u \in \mathscr{V}$. Then for all $u \in \mathscr{V}$, $f^m_u: U_u \to V_v$ is a polynomial-like mapping with the unique critical point $0$. However, this is impossible for $u \in \mathscr{V} \cap \mathcal{H}_0$.
\end{proof}

As for the hyperbolic components of type C, we have the similar result. 
\begin{lemma}\label{lm: maps on the boundary of U}
    Let $\mathcal{U}$ be a component of $\mathcal{H}_m$, $m \ge 1$, $v \in \partial \mathcal{U}$. Then $\partial B_v$ either contains a parabolic cycle or $f^m_v(-2v)$.
\end{lemma}

At the end of this section, we study the dynamics of $f_v$ with $v \in i\mathbb{R} \cup \mathbb{R}$.
\begin{lemma}\label{lm: maps on the imaginary axis}
    Let $v \in i\mathbb{R}$. Then there exists $y_0>0$ such that the following trichotomy holds.
    \begin{enumerate}
        \item for $0<|v|<y_0$, $v \in \mathcal{H}_0$, so that $f_v$ is hyperbolic;
        \item for $|v|=y_0$, $f_v$ is critically finite. More precisely, $2v$ is a repelling fixed point of $f_v$, and $f_v(-2v)=2v$;
        \item for $|v|>y_0$, the orbit of $-2v$ escapes to $\infty$. In other words, $-2v \in I(f_v)$. 
    \end{enumerate}
\end{lemma}
\begin{proof}
    Since the parameter space of $f_v$ is symmetric about the origin, we may assume that $\mathrm{Im} \; v>0$. For $y>0$, consider the equation $2iy=f_{iy}(2iy)$. This equation has the unique solution $y_0=\log(\sqrt{2}+1)$ on $\mathbb{R}_+$. This implies that for the parameter $v_0=i\log(\sqrt{2}-1)$, $-2v_0$ is mapped to a fixed point $2v_0$ of $f_{v_0}$. 

    Note that for $v \in i\mathbb{R}$, $f_v$ maps the imaginary axis to itself. For $0<|v|<y_0$, $|f_v(-2v)|<|-2v|$. As a result, $f^n_v(-2v) \to 0$ as $n \to \infty$, i.e. $v \in \mathcal{H}_0$. Similarly, for $|v|>v_0$, $f^n_v(-2v) \to +\infty$ as $n \to \infty$, i.e. $-2v \in I(f_{v})$. 
\end{proof}

\begin{corollary}\label{cor: closure of hyperbolic components of type C are disjoint from the real axis}
    Let $\mathcal{U}$ be a hyperbolic component of type C. Then $\overline{\mathcal{U}} \cap i\mathbb{R}=\varnothing$. 
\end{corollary}
\begin{proof}
    By Lemma \ref{lm: maps on the imaginary axis}, the only possibility is that $v_0 \in \partial \mathcal{U}$. Moreover, $\overline{\mathcal{U}}$ cannot intersect with the upper and lower half plane at the same time. Suppose that $\mathcal{U} \subset \mathcal{H}_m$, $m \ge 1$. By Lemma \ref{lm: maps on the boundary of U} and the continuity of $\partial B_v$ with respect to $v$ (similar to the proof of \ref{lm: k_v is constant}), $m$ is the smallest integer such that $f^j_{v_0}(-2v_0 ) \in \partial B_{v_0}$. However, since $v_0 \in \partial \mathcal{H}_0$, by Lemma \ref{lm: maps on the boundary of H_0}, $-2v_0 \in \partial B_{v_0}$, a contradiction.  
\end{proof}

\begin{corollary}\label{cor: local connectivity of the boundary of H_0 at v_0}
    $\partial \mathcal{H}_0$ is locally connected at $\pm v_0$. 
\end{corollary}
\begin{proof}
    Recall that there is a degree-$2$ covering map $\Phi_0: \mathcal{H}_0 \to \mathbb{D}^*$ defined by \eqref{eq: Phi_0}. It is easy to see that $\Phi_0((0,v_0))=(0,i)$. By the symmetry of $\mathcal{H}_0$, $\Phi_0((-v_0, 0))=(0,i)$. For $\epsilon >0$ sufficiently small, define 
    $$\mathcal{L}_0(\epsilon)=\overline{\Phi^{-1}_0(\{re^{i\theta}: 1-\epsilon<r<1, |\theta-\pi/2|<\epsilon\})}, $$
    which are the closed connected neighborhoods of $\pm v_0$. Let 
    $$\mathcal{L}_0=\bigcap_{n \ge 1} \mathcal{L}_0(\frac{1}{n}), $$
    and $\mathcal{L}^{\pm}_0$ be the component of $\mathcal{L}_0$ containing $\pm v_0$ respectively. Then it will be sufficient to show that $\mathcal{L}^{\pm}_0$ are singeltons. 

    Obviously, $\mathcal{L}^{\pm}_0 \subset \partial \mathcal{H}_0 \cap \mathbb{R}$. By Lemma \ref{lm: maps on the imaginary axis}, $\mathcal{L}^{\pm}_0$ consists of single points. 
\end{proof}

\begin{lemma}\label{lm: maps on the real axis}
    There exists $v_1 >0$, such that $f_{v_1}$ has a parabolic fixed point $z_0$, $f'_{v_1}(z_0)=1$. Moreover, $\mathcal{H}_0 \cap \mathbb{R}=(-v_1, 0) \cup (0,v_1)$
\end{lemma}
\begin{proof}
    By the symmetry of $\mathcal{H}_0$, we may assume $v>0$. Consider the equations
    \begin{equation}\label{eq: f_v(x)=x, f'_v(x)=1}
        \begin{cases}
            v(\cos x-1) = x \\
            -v \sin x =1
        \end{cases}.
    \end{equation}
    
    Let $z_0 <0$ be the solution to $\tan (z/2)=z$ such that $|z_0|$ is the smallest among all the solutions, and $v_1=-1/\sin z_0$. Then $(v_1, z_0)$ is a solution to \eqref{eq: f_v(x)=x, f'_v(x)=1}. It is easy to see that $\pi/2<|z_0|<\pi$ so that $\sin^2 v_0 \le v^2_0 < 1$.Thus for $0<v<v_1$, $|f_v(-2v)|<|-2v|$ and $f^n_v(-2v) \to 0$ as $n \to \infty$. So $(0,v_1) \subset \mathcal{H}_0 \cap \mathbb{R}$. Obviously, $v_1 \in \partial \mathcal{H}_0$. Finally, we show that $(v_1, +\infty)$ cannot intersect with $\mathcal{H}_0$. Since $\mathcal{H}_0$ is doubly connected, $\partial \mathcal{H}_0$ has no bounded components other than $\{0\}$. Therefore, there is a Jordan arc $\gamma \subset \mathbb{C} \setminus \mathcal{H}_0$ in the upper half plane, connecting $v_1$ to $\infty$. By the symmetry of $\mathcal{H}_0$, the conjugacy $\bar{\gamma}$ of $\gamma$ is also a Jordan arc connecting $v_1$ to $\infty$. Hence $\gamma \cup \bar{\gamma}$ is a Jordan curve seperating $\mathcal{H}_0$ from $(v_1, +\infty)$. By the connectedness of $\mathcal{H}_0$, $\mathcal{H}_0 \cap \mathbb{R}_+=(0,v_1)$.  
\end{proof}

\subsection{Construction of para-puzzles}\label{ss: construction of para-puzzles}
In this section, we will construct para-puzzles for the cosine family. 

Recall that $g^v_{\underline{s}}$ is a dynamic ray with the external address $\underline{s}$. $B_v$ is the immediate basin of $0$, $\phi_v: U_v \to D(0,r_v)$ is the B\"ottcher mapping defined in a neighborhood of the superattracting fixed point $0$, satisfying $\phi_v \circ f^v(z)=\phi_v(z)^2$. If $v \notin \mathcal{H}_0$, then $\phi_v$ is well-defined on the entire $B_v$, and the internal rays $R^v_0(\theta)$ are well-defined for $\theta \in \mathbb{R}/\mathbb{Z}$. 

By the Implicit Function Theorem, the internal ray (or the dynamic ray resp.) together with its repelling landing point is stable under small perturbations of parameters. 

\begin{lemma}\label{lm: stability of rays and the repelling landing points}
    Suppose that $v_0 \notin \mathcal{H}_0$ such that $R^v_0(\theta)$ (or $g^v_{\underline{s}}$ resp.) is well-defined and lands at a repelling periodic point $z_0$, whose period is $k \ge 1$. Then the ray $R^v_0(\theta)$ (or $g^v_{\underline{s}}$ resp.) together with $z_0$ is stable in the following sense:

    There is a neighborhood $\mathcal{N}$ of $v_0$ and an analytic function $z(v)$ on $\mathcal{N}$ such that for all $v \in \mathcal{N}$, 
    \begin{enumerate}
        \item $z(v)$ is a repelling periodic point of period $k$, and $z(v_0)=z_0$;
        \item $R^v_0(\theta)$ (or $g^v_{\underline{s}}$ resp.) is well-defined;
        \item $R^v_0(\theta)$ (or $g^v_{\underline{s}}$ resp.) lands at $z(v)$.
    \end{enumerate}
\end{lemma}

Let $\mathcal{U}$ be a component of $\hat{\mathcal{H}}$, $v_0 \in \partial \mathcal{U}$. In the dynamical plane of $f_{v_0}$, there is an $f_{v_0}$-invariant graph $I^{v_0}_0(\theta_0)$ consists of internal rays $R^{v_0}_0(\theta_j) \subset B_{v_0}$ and dynamic rays $g^{v_0}_{\underline{s}_j} \subset I(f_{v_0})$, $j=0, \ldots, q-1$. The angles $\theta_j=2^j/(2^q-1)$ such that $R^{v_0}_{0}(\theta_j)$ land at repelling periodic points of period $q$. $\underline{s}_j$, $j=0, \ldots, q-1$ are the external addresses such that $g^{v_0}_{\underline{s}_j}$ lands at the same point as $R^{v_0}_0(\theta_j)$, so that $g^{v_0}_{\underline{s}_j}$ are also of period $q$. And from the invariant graph $I^{v_0}_0(\theta)$, we can define a puzzle system $\{P^{v_0}_n(z)\}_{n \ge 0}$, satisfying 
\begin{enumerate}
    \item every $P^{v_0}_n(z)$, $n \ge 1$ contains at most one critical point of $f_{v_0}$;
    \item $f_{v_0}(P^{v_0}_n(z))=P^{v_0}_{n-1}(f_{v_0}(z))$, $n \ge 1$;
    \item two puzzle pieces $P^{v_0}_n(z_1)$, $P^{v_0}_n(z_2)$ are either disjoint or equal.
\end{enumerate}
One can refer to \cite{QW25} for more details. 

Let $\mathcal{E}_0(t)=\Phi^{-1}_0(t e^{2\pi i [0,1)})$ and $\mathcal{E}_{\mathcal{U}}(t)=\Phi^{-1}_{\mathcal{U}}(te^{2\pi i [0,1)})$, where $\mathcal{U}$ is a hyperbolic component of type C. Let $\mathcal{E}_n$ be the collection of $(\mathcal{U}, t)$, satisfying that 
\begin{enumerate}
    \item $\mathcal{U}$ is a component of $\mathcal{H}_m$, $0 \le m \le n$;
    \item $t^{2^{n-m}}=t_0$. 
\end{enumerate}
Let $\mathcal{X}_n$ be the component of $\mathbb{C} \setminus \left(\bigcup_{(\mathcal{U}, t) \in \mathcal{E}_n} \mathcal{E}_{\mathcal{U}}(t)\right)$ containing $v_0$. 

Now define a graph $\mathcal{I}_0(\theta_0)$ in the parameter plane by 
$$\bigcup_{j=0}^{q-1} \overline{\mathcal{R}_0(\theta_j) \cup G_{\underline{s}_j}}. $$
For $n \ge 1$, let $\mathcal{D}_n$ be the collection of $(\mathcal{U}, \theta)$, where $\mathcal{U}$ is a component of $\mathcal{H}_m$, $m=0, \ldots, n$, $2^{n-m}\theta \in \{\theta_0, \ldots, \theta_{q-1}\}$, and let $\mathcal{S}_n=\{\underline{s} \in \Sigma: \sigma^n(\underline{s})= \underline{s}_j, j=0, \ldots, q-1\}$. Define the parameter graph of level $n$ as 
$$\mathcal{I}_n(\theta_0)=\left(\bigcup_{(\mathcal{U}, \theta) \in \mathcal{D}_n} \overline{\mathcal{R}_{\mathcal{U}}(\theta)}\right) \cup \left(\bigcup_{\underline{s} \in \mathcal{S}_n} G_{\underline{s}}\right). $$
The components of $(\mathbb{C} \setminus \mathcal{I}_n(\theta_0)) \cap \mathcal{X}_n$ are called the para-puzzle pieces of depth $n$. For $v  \notin \bigcup_{n \ge 0} \mathcal{I}_n(\theta_0)$, denote by $\mathcal{P}_n(v)$ to be the para-puzzle piece of depth $n$ containing $v$. It is easy to see that $\mathcal{I}_n(\theta_0) \subset \mathcal{I}_{n+1}(\theta_0)$, thus $\mathcal{P}_{n+1}(v) \subset \mathcal{P}_n(v)$. If $v_0 \in \partial \mathcal{U}$ is given, for the sake of simplicity, write $P^{v_0}_n=P^{v_0}_n(-2v_0)$, $I^{v_0}_n=I^{v_0}_n(\theta_0)$, $\mathcal{P}_n=\mathcal{P}_n(v_0)$, and $\mathcal{I}_n=\mathcal{I}_n(\theta_0)$.

\subsection{Holomorphic motion of para-puzzle pieces}\label{ss: holomorphic motion}
\begin{definition}\label{def: holomorphic motion}
    Let $\Lambda \subset \mathbb{C}$ be a domain, $\lambda_0 \in \Lambda$, $X$ be a subset of $\hat{\mathbb{C}}$. A holomorphic motion of $X$ parameterized by $\Lambda$ with the base point $\lambda_0$ is a map $h: \Lambda \times X \to \hat{\mathbb{C}}$, satisfying 
    \begin{enumerate}
        \item for every $x \in X$, $h(\lambda_0, x)=x$;
        \item for every $\lambda \in \Lambda$, $h^{\lambda}=h(\lambda, \cdot): X \to \hat{\mathbb{C}}$ is injective;
        \item for every $x \in X$, $h_x=h(\cdot, x): \Lambda \to \hat{\mathbb{C}}$ is analytic. 
    \end{enumerate} 
\end{definition}

\begin{lemma}\label{lm: holomorphic motion of parameter rays}
    Let $v_0$, $\mathcal{U}$ be as defined in \ref{ss: construction of para-puzzles}. Let $\Omega$ be the set of $v$ such that $g^v_{\sigma^j(\underline{s})}$, $j \ge 0$ are well-defined and land at repelling periodic points. Then 
    \begin{enumerate}
        \item $\Omega$ is non-empty and open;
        \item there is a holomorphic motion of $\bigcup_{j \ge 0} g^{v_0}{\sigma^j(\underline{s}_0)}$ parrameterized by $v \in \Omega$;
        \item $\partial \Omega= \bigcup_{j \ge 0} \overline{G_{\sigma^j(\underline{s}_0)}}$, which contains no isolated points. 
    \end{enumerate} 
\end{lemma}
\begin{proof}
    1. Obviously, $v_0 \in \Omega$. By Lemma \ref{lm: stability of rays and the repelling landing points}, $\Omega$ is open. 

    2. Define $h: \Omega \times \overline{g^{v_0}_{\underline{s}_0}} \to \mathbb{C}$ as $h(v, g^{v_0}_{\underline{s}_0}(t))=g^v_{\underline{s}_0}(t)$. It is clear that $h^{v_0}=\mathrm{id}$. $h^{v}$ is injective otherwise the dynamic ray $g^v_{\underline{s}_0}$ will intersect either itself or its landing point, which are both impossible. For every $z=g^{v_0}_{\underline{s}_0}(t)$ with $t>0$ fixed, $h_z$ is analytic. $h$ extends analytically to the landing point $z_0(v)$ by Lemma \ref{lm: stability of rays and the repelling landing points}. 

    3. Let $v_1 \in \partial \Omega$. First we show that $\partial \Omega$ contains no isolated points. Otherwise, let $v_1$ be such a point. Then there is a neighborhood $V$ of $v_1$, such that $V \setminus \{v_1\} \subset \Omega$. For $v_1$, there is some $\underline{s} \in \{\sigma^j(\underline{s}_0): j \ge 0\}$, either $-2v_1 \in g^{v_1}_{\underline{s}}$, or it lands at a parabolic periodic points. If $-2v_1 \in g^{v_1}_{\underline{s}}$, then $v_1 \in G_{\underline{s}}$ in the parameter plane. Thus for $v \in V \cap G_{\underline{s}}$, $-2v \in g^{v}_{\underline{s}}$, contradicting that $g^v_{\underline{s}}$ is well-defined for $v \in V \setminus \{v_1\}$. If $g^{v_1}_{\underline{s}}$ lands at a parabolic periodic point $z(v_1)$ of period $p \ge 1$, then the multiplier $(f^p_{v_1}(z(v_1)))'$ of $z(v_1)$ defines an analytic function in $V$. It attains its minimal modulus at $v_1$ since $z(v)$ is repelling for all $v \in V$ except $v_1$, contradicting to the Maximum Modulus Principle. 

    Again let $v_1 \in \partial \Omega$. For $v_1$, there is some $\underline{s} \in \{\sigma^j(\underline{s}_0): j \ge 0\}$, either $-2v_1 \in g^{v_1}_{\underline{s}}$, or it lands at a parabolic periodic points. Thus either $v_1 \in G_{\underline{s}}$, or it belongs to the set $\mathrm{Par}_{q}(1)=\{v \in \mathbb{C}: \text{there exists } z_v \in \mathbb{C}, f^q_v(z_v)=z_v, (f^q_v)'(z_v)=1\}$, which is discrete in $\mathbb{C}$. Since $\partial \Omega$ contains no isolated points, $\mathrm{Par}_{q}(1) \cap \partial \Omega$ is contained in $\overline{\bigcup_{j \ge 0} G_{\sigma^j(\underline{s}_0)}}$. 
\end{proof}

Similarly, we have the following 
\begin{lemma}\label{lm: holomorphic motion of para-internal rays}
    Let $v_0$, $\mathcal{U}$ be as defined in \ref{ss: construction of para-puzzles}. Let $\Omega'$ be the set of $v$ such that $R^v_{0}(2^j \theta_0)$, $j \ge 0$ are well-defined and land at repelling periodic points. Then 
    \begin{enumerate}
        \item $\Omega'$ is non-empty and open;
        \item there is a holomorphic motion of $\bigcup_{j \ge 0} R^{v_0}_{0}(2^j \theta_0)$ parrameterized by $v \in \Omega'$;
        \item $\partial \Omega'=\bigcup_{j \ge 0} \overline{R^v_{0}(2^j \theta_0)}$, which contains no isolated points. 
    \end{enumerate} 
\end{lemma}

By Lemma \ref{lm: holomorphic motion of parameter rays} and \ref{lm: holomorphic motion of para-internal rays}, we have 
\begin{corollary}\label{cor: holomorphic motion of para-graph of depth 0}
    $\mathcal{P}_0$ is contained in the component of $\Omega \cap \Omega'$ containing $v_0$. Thus there is a holomorphic motion of $I^{v_0}_0$ parameterized by $\mathcal{P}_0$.  
\end{corollary}

\begin{corollary}\label{cor: points on the para-graph}
    For $n \ge 1$, $v \in \mathcal{I}_n \cap \mathcal{P}_{0}$ if and only if $-2v \in I^{v}_n$. 
\end{corollary}

The holomorphic motion in Corollary \ref{cor: holomorphic motion of para-graph of depth 0} can be lifted to para-graphs of higher depths.
\begin{lemma}\label{lm: holomorphic motion of para-graphs}
    Let $\mathcal{U}$ be a component of $\hat{\mathcal{H}}$, $v_0 \in \partial \mathcal{U}$, and $n \ge 0$. Let $\mathcal{P}_n$, $I^{v_0}_n$ be defined as in \S \ref{ss: construction of para-puzzles}. Then there is a holomorphic motion
    \begin{equation}\label{eq: H_n}
        H_n: \mathcal{P}_n \times I^{v_0}_{n+1} \to \mathbb{C}
    \end{equation}
    satisfying:
    \begin{enumerate}
        \item $H^{v}_n(I^{v_0}_{n+1})=I^v_{n+1}$;
        \item $H_{n}$ and $H_{n-1}$ coincide on $\mathcal{P}_n \times I^{v_0}_n$;
        \item for $z \in I^{v_0}_{n+1}$ and $v \in \mathcal{P}_n$, $H^v_{n-1} \circ f_{v_0}(z)=f_v \circ H^{v}_{n}(z)$. 
    \end{enumerate}  
\end{lemma}

\begin{proof}
    Let $H_0: \mathcal{P}_0 \times I^{v_0}_0 \to \mathbb{C}$ be the holomorphic motion defined in Corollary \ref{cor: holomorphic motion of para-graph of depth 0}. We define $H_n$ by induction on $n \ge 0$. Suppose that $H_n: \mathcal{P}_n \times I^{v_0}_n \to \mathbb{C}$ is defined. For $v \in \mathcal{P}_{n+1}$, consider the following diagram:
    \[
    \xymatrix{
    I^{v_0}_{n+1} \ar@{-->}[r]^{H^v_{n+1}} \ar@{->}[d]_{f_{v_0}} & I^v_{n+1} \ar@{->}[d]^{f_v} \\
    I^{v_0}_n \ar@{->}[r]^{H^v_n}              & I^v_n
    }
    \]
    Since $v \in \mathcal{P}_{n+1}$, $I^v_{n+1}$ does not touch the critical points. Thus $H^v_n$ can be lifted to a homeomorphism $H^v_{n+1}$. It is easy to check that 1-3 in Lemma \ref{lm: holomorphic motion of para-graphs} are satisfied. 
\end{proof}

\subsection{transfer from the dynamical plane to the parameter plane}\label{ss: transfer from the dynamical plane to the parameter plane}
Let $H_n$ be the holomorphic motion in Lemma \ref{lm: holomorphic motion of para-graphs}. It connects the parameter plane with the dynamical plane of $f_{v_0}$. 
\begin{lemma}\label{lm: transfer from dynamical plane to parameter plane}
    For $n \ge 0$, there is a homeomorphism 
    \begin{equation}\label{eq: h_n}
        \begin{split}
            h_n: \mathcal{P}_n \cap \mathcal{I}_{n+1} & \to P^{v_0}_n \cap I^{v_0}_{n+1} \\
            v & \mapsto (H^v_n)^{-1}(-2v)
        \end{split}
    \end{equation} 
\end{lemma}
\begin{proof}
    Let $v \in \mathcal{P}_n(v_0) \cap \mathcal{I}_{n+1}$. By Corollary \ref{cor: points on the para-graph}, $-2v \in I^v_{n+1}$. Thus $h_n(v)$ is well-defined and belongs to $I^{v_0}_{n+1}$. For $v_0$, by the definition, $-2v_0 \in P^{v_0}_n$. Thus $-2v$ belongs to the puzzle piece bounded by $H^v_n(\partial P^{v_0}_n)$. Since $-2v$ and $I^v_{n+1}$ are disjoint, $h_n(-2v) \in P^{v_0}_n$. 
    
    \textbf{Injectivity of $h_n$}: 
    
    First suppose that $v=G_{\underline{s}}(t)$, $t \ge 0$. Then $-2v=g^v_{\underline{s}}(t)$, and $h_n(v)=g^{v_0}_{\underline{s}}(t)$. Thus $h_n$ is injective on the parameter rays.

    Then consider parameter internal rays. Suppose that $v_1=\mathcal{R}_{\mathcal{U}_1}(\theta_1)(t_1)$, $v_2=\mathcal{R}_{\mathcal{U}_2}(\theta_2)(t_2)$, and $h_n(v_1)=h_n(v_2)=z_0$. Suppose that $\mathcal{U}_j$ is a component of $\mathcal{H}_{m_j}$, $j=1,2$. Then $f^{m_j}_{v_j}(-2v_j)=R_{v_j}(\theta_j)(t_j)$, $j=1,2$, where $R_{v}(\theta)$ is the internal ray in $B_v$. For $j=1,2$, $f^{m_j}_{v_j}(-2v_j)$ is the first point in the orbit of $-2v_j$ entering the immediate basin $B_{v_j}$. Passing to the dynamical plane of $f_{v_0}$ by the holomorphic motion $H^{v_j}_n$, $f^{m_j}_{v_0}(z_0)$ is the first point in the orbit of $z_0$ entering $B_{v_0}$. Thus $m_1=m_2=m$, $t_1=t_2$.

    If $m=0$, then $v_2=\pm v_1$. Since they belong to the same puzzle piece $\mathcal{P}_n$, $v_1=v_2$. 

    If $m \ge 1$, then $f^{m-1}_{v_j}(-2v_j) \in B_{v_j}+2k_j\pi$, where $k_j \in \mathbb{Z}$ depends on $\mathcal{U}_j$ by Lemma \ref{lm: k_v is constant}. Again by the holomorphic motion $H^{v_j}_n$, $k_1=k_2$. So $k_1=k_2=k$ and $\mathcal{U}_1=\mathcal{U}_2=\mathcal{U}$. By Proposition \ref{prop: Phi_U is a conformal isomorphism}, $\Phi_{\mathcal{U}}: \mathcal{U} \to \mathbb{D}$ is a conformal isomorphism, $v_1=\mathcal{R}_{\mathcal{U}}(\theta)(t)=v_2$. 

    Finally, suppose that $v_1$ and $v_2$ are the landing points of either parameter rays or internal rays. 

    If $v_1$ and $v_2$ are the landing points of parameter rays $G_{\underline{s}_1}$ and $G_{\underline{s}_2}$ respectively, then $g^{v_0}_{\underline{s}_1}$ and $g^{v_0}_{\underline{s}_2}$ land at a common point. By the construction of $I^{v_0}_n$, $\underline{s}_1=\underline{s}_2$ so that $v_1=v_2$. The similar argument works if $v_1$ and $v_2$ are the landing points of parameter internal rays. 
    
    Now suppose that $v_1$ is the landing point of $\mathcal{R}_{\mathcal{U}}(\theta)$, $v_2$ is the landing point of $G_{\underline{s}}$. Then $R_{v_0}(\theta)$ and $g^{v_0}_{\underline{s}}$ land at a common point. Pushing forward to $v_1$-plane by $H^{v_1}_n$, $R_{v_1}(\theta)$ and $g^{v_1}_{\underline{s}}$ land at a common point, which is exactly $-2v_1$. So $v_1$ is the landing point of $G_{\underline{s}}$, i.e. $v_1=v_2$. 
    
    \textbf{Surjectivity of $h_n$}:

    For $z=g^{v_0}_{\underline{s}}(t)$, let $v=G_{\underline{s}}(t)$. Then $-2v =g^v_{\underline{s}}(t)$ so that $h_n(v)=z$. 

    If $z=R_{v_0}(\theta)(t)$, $0 \le t \le 1$, then $R_{v_0}(\theta)$ lands at some point $z_1$. And there is a unique dynamic ray $g^{v_0}_{\underline{s}}$ in $I^{v_0}_{n+1}$ landing at $z_1$. Consider $G_{\underline{s}}$. It lands at some $v_1$ by Proposition \ref{prop: landing theorem for parameter rays}. By the construction of $\mathcal{I}_{n+1}$, there is a unique $\mathcal{R}_{\mathcal{U}}(\theta')$ also landing at $v_1$. Thus $\overline{G_{\underline{s}}\cup \mathcal{R}_{\mathcal{U}}(\theta')}$ is mapped to $\overline{g^{v_0}_{\underline{s}} \cup R_{v_0}(\theta)}$, and there is $v \in \mathcal{R}_{\mathcal{U}}(\theta')$ such that $h_n(v)=z$. 

    Continuity of $h_n$ and $h^{-1}_n$ is obvious.
\end{proof}

\begin{remark}\label{rmk: extending to euqipotentials}
    The homeomorphism $h_n$ can be easily extended the equipotentials $\mathcal{E}_{\mathcal{U}}(t)$, $(\mathcal{U},t) \in \mathcal{E}_{n+1}$, where $\mathcal{E}_{n+1}$ is defined in \S \ref{ss: construction of para-puzzles}. 
\end{remark}

Let $\mathcal{U}$ be a component of $\hat{\mathcal{H}}$ and $v_0 \in \partial \mathcal{U}$. In \cite{QW25}, the authors modified the unbounded puzzle pieces $P^{v_0}_{n}$ to bounded ones $\hat{P}^{v_0}_n$, $n \ge 1$. Moreover, the following holds. 
\begin{lemma}\label{lm: non-degenerated annuli}
    Let $\mathcal{U}$ be a component of $\hat{\mathcal{H}}$ and $v_0 \in \partial \mathcal{U}$. Then there exists a sequence of puzzle pieces $\{\hat{P}^{v_0}_{n_k}\}_{k \ge 0}$, such that 
    \begin{enumerate}
        \item $-2v_0 \in \hat{P}^{v_0}_{n_k}$;
        \item let $A^{v_0}_{n_k}=\hat{P}^{v_0}_{n_k} \setminus \overline{\hat{P}^{v_0}_{n_k+1}}$, then $f^{n_k-n_0}_{v_0}: A^{v_0}_{n_k} \to A^{v_0}_{n_0}$ is a covering map;
        \item either $\sum_{k \ge 0} \mathrm{mod} \; A^{v_0}_{n_k} =\infty$ so that $\bigcup_{n \ge 1} \overline{\hat{P}^{v_0}_{n}}=\{-2v_0\}$, or there is $p \ge 1$ such that $n_{k+1}=n_k+p$ and $f^p_{v_0}: \hat{P}^{v_0}_{n_{k+1}} \to \hat{P}^{v_0}_{n_k}$ is a renormalization of $f_{v_0}$, where $p$ is called the period of the renormalization. 
    \end{enumerate}
\end{lemma}

For $v \in \mathcal{P}_n$, let $t_0>0$ be as in \S \ref{ss: construction of para-puzzles}. Define $E^v_n$ be the collection of $(U,t)$, where $U$ is a component of $f^{-m}_{v}(B_v)$, $0 \le m \le n$, $t^{2^{n-m}}=t_0$. Since $v \in \mathcal{P}_n$, the equipotential $E^v_{B_{v}}(t)$ in $B_v$ can be defined for $t \le t_0^{1/2^n}$. And it can be puled back to the components $U$ of $f^{-m}_v(B_v)$, $0 \le m \le n$. Define 
$$D^{v}_n = \bigcup_{(U,t) \in E^v_n} E^v_U(t).$$
Then similar to Lemma \ref{lm: holomorphic motion of para-graphs}, there is a holomorphic motion $H_n: \mathcal{P}_n \times D^{v_0}_n \to \mathbb{C}$, with $H^v_n(D^{v_0}_n)=D^v_n$. Combining this with Lemma \ref{lm: holomorphic motion of para-graphs}, we obtain a holomorphic motion $H_n: \mathcal{P}_n \times (\partial P^{v_0}_n \cup \partial P^{v_0}_{n+1}) \to \mathbb{C}$. 

Let $n \ge 1$. By Slodkowski's Theorem, $H_n$ extends to a holomorphic motion $\tilde{H}_n: \mathcal{P}_n \times \hat{\mathbb{C}} \to \hat{\mathbb{C}}$. In particular, the puzzle pieces $\hat{P}^{v_0}_n$ moves holomorphically. Define $\hat{P}^v_{n}=H^v_n(\hat{P}^{v_0}_n)$, $A^v_n=\hat{P}^v_n \setminus \overline{\hat{P}^v_{n+1}}$. $A^v_n$ is indeed an annulus since $H^v_n$ is injective. 

This allows us to define a `cut-off' of para-puzzle pieces. $\{v \in \mathcal{P}_n: -2v \in \hat{P}^v_n\}$ is non-empty since it contains $v_0$. Define 
$\hat{\mathcal{P}}_n$ to be the component of $\{v \in \mathcal{P}_n: -2v \in \hat{P}^v_n\}$ containing $v_0$, and $\mathcal{A}_n=\hat{\mathcal{P}}_n \setminus \overline{\hat{\mathcal{P}}_{n+1}}$. We do not know whether $\mathcal{A}_n$ are non-degenerated annuli yet. But Lemma \ref{lm: non-degenerated para-annuli} ensures this. 

\begin{lemma}\label{lm: non-degenerated para-annuli}
    $\mathcal{A}_{n_k}$, $k \ge 0$ are non-degenerated annuli, which surround $v_0$. Moreover, there exists $K \ge 1$, such that 
    $$\frac{1}{K} \mathrm{mod} \; A^{v_0}_{n_k} \le \mathrm{mod} \; \mathcal{A}_{n_k} \le K \mathrm{mod} \; A^{v_0}_{n_k}$$
    for all $k \ge 1$. 
\end{lemma}
\begin{proof}
    Let $n \in \{n_k: k \ge 0\}$. Define 
    \begin{equation}\label{eq: tilde h_n}
        \begin{split}
            \tilde{h}_n: \mathcal{A}_n \cup \partial \hat{\mathcal{P}}_{n+1} & \to A^{v_0}_n \cup \partial \hat{P}^{v_0}_{n+1} \\
            v & \mapsto (\tilde{H}^v_n)^{-1}(-2v)
        \end{split}.
    \end{equation}

    For $v \in \partial \mathcal{P}_{n+1}$, in the dynamical plane of $f_v$, $-2v \in \partial \hat{P}^v_{n+1}$. Thus $\tilde{h}_n(v)=(\tilde{H}^v_n)^{-1}(-2v) \in \partial \hat{P}^{v_0}_{n+1}$. For $v \in \mathcal{A}_n$, $-2v \in \hat{P}^v_n \setminus \overline{\hat{P}^v_{n+1}}$. Thus $\tilde{h}_n(v) \in \hat{P}^{v_0}_{n} \setminus \overline{\hat{P}^{v_0}_{n+1}}=A^{v_0}_{n}$. Therefore, $\tilde{h}_n$ is well-defined and sends $\mathcal{A}_n \cup \partial \hat{\mathcal{P}}_{n+1}$ to $A^{v_0}_n \cup \partial \hat{P}^{v_0}_{n+1}$. It is easy to see that $\tilde{h}_n$ extends $h_n$, where $h_n: \mathcal{P}_n \cap \mathcal{I}_{n+1} \to P^{v_0}_n \cap I^{v_0}_{n+1}$ is defined by \eqref{eq: h_n}. 

    By $\lambda$-lemma (\cite{MSS83}), $H^v_n$ is a $K_n(v)$-quasiconformal mapping, where the dilatation $K_n(v)$ depends on $v \in \mathcal{P}_n$. By \cite{DH85}, IV. 3. Lemma (also \cite{Roe00}, Proposition 3.2), the transfer map $\tilde{h}_n$ is $K_n$-quasi-regular, with $K_n \le \sup\{\tilde{K}_n(v): v \in \hat{\mathcal{P}}_n\}$, where $\tilde{K}_n(v)$ is the dilatation of $\tilde{H}^v_n$. Note that $\tilde{h}_n$ is a proper map. Note that $v \in \mathcal{I}_{n}$ if and only if $\tilde{h}_n(v) \in I^{v_0}_{n}$. Thus $\tilde{h}_n(v)$ describes $\partial \hat{P}^{v_0}_n$ once as $v$ describes $\partial \hat{\mathcal{P}}_n$ once, since $\tilde{h}_n$ coincides with $h_n$ on $\partial \mathcal{I}_{n+1}$, which is a homeomorphism by Lemma \ref{lm: transfer from dynamical plane to parameter plane}. By the Argument Principle, $\tilde{h}_n: \mathcal{A}_n \cup \partial \hat{\mathcal{P}}_{n+1} \to A^{v_0}_n \cup \partial \hat{P}^{v_0}_{n+1}$ is also a homeomorphism, and $\mathcal{A}_n=(\tilde{h}_n)^{-1}(A^{v_0}_n)$ is an annulus. $\mathcal{A}_n$, $n \ge n_0$ surrounds $v_0$ since $v_0$ belongs to every $\hat{\mathcal{P}}_n$.

    Finally, we prove that $K_n$ is bounded from above. By its definition, $H^v_n$ is obtained by lifting $H^v_{0}$. Thus $K_n(v)=K_0(v)$. The extension by Slodkowski's Theorem does not change the dilatation. Hence $\tilde{K}_n(v)=K_n(v)=K_0(v)$. Moreover, for $k \ge 1$, $\hat{\mathcal{P}}_{n_k} \Subset \hat{\mathcal{P}}_{n_0+1} \subset \mathcal{P}_{n_0}$. Therefore, $K_n \le \sup\{K_n(v): v \in \hat{\mathcal{P}_n}\} \le \sup\{K_{n_0+1}(v): v \in \hat{\mathcal{P}}_{n_0+1}\}=:K<+\infty$. This gives that for $k \ge 1$, the dilatation of $\tilde{h}_n: \mathcal{A}_n \to A_n$ is bounded from above. Thus  
    $$\frac{1}{K} \mathrm{mod} \; A^{v_0}_{n_k} \le \mathrm{mod} \; \mathcal{A}_{n_k} \le K \mathrm{mod} \; A^{v_0}_{n_k}.$$
\end{proof}

\begin{corollary}\label{cor: local connectivity at non-renormalizable parameters}
    Let $\mathcal{U}$ be a component of $\hat{\mathcal{H}}$, $v_0 \in \partial \mathcal{U}$. Suppose that $f_{v_0}$ is not renormalizable. Then $\partial \mathcal{U}$ is locally connected at $v_0$. 
\end{corollary}
\begin{proof}
    For $v_0$, there are two sequences of annuli $\{A^{v_0}_{n_k}\}_{k \ge 0}$ and $\{\mathcal{A}_{n_k}\}$, satisfying
    \begin{enumerate}
        \item $A^{v_0}_{n_k}$, $k \ge 0$ surrounds $-2v_0$ in the dynamical plane of $f_{v_0}$;
        \item $\mathcal{A}_{n_k}$ surrounds $v_0$ in the parameter plane. 
    \end{enumerate}
    By Lemma \ref{lm: non-degenerated para-annuli}, there is $K \ge 1$ such that 
    $$\frac{1}{K} \mathrm{mod} \; A^{v_0}_{n_k} \le \mathrm{mod} \; \mathcal{A}_{n_k} \le K \mathrm{mod} \; A^{v_0}_{n_k}.$$

    Since $f_{v_0}$ is not renormalizable, by Lemma \ref{lm: non-degenerated annuli}, $\sum_{k \ge 0} \mathrm{mod} \; A^{v_0}_{n_k}=+\infty$. Thus $\sum_{k \ge 0} \mathrm{mod} \; \mathcal{A}_{n_k}=+\infty$. By Gr\'otzsch's inequality, $\bigcap_{n \ge n_0} \overline{\hat{\mathcal{P}}_n}=\{v_0\}$. Therefore, $\{\hat{\mathcal{P}}_n \cap \partial \mathcal{U}\}_{n \ge 0}$ forms a topological basis of $\partial \mathcal{U}$ at $v_0$. Obviously, $\hat{\mathcal{P}}_n \cap \partial \mathcal{U}$, $n \ge 0$ are connected. Thus $\partial \mathcal{U}$ is locally connected at $v_0$. 
\end{proof}

\section{Local connectivity of the boundaries of components in $\hat{\mathcal{H}}$}\label{sec: local connectivity of the boundaries of hyperbolic components}

\subsection{Renormalizable parameters on the boundaries of hyperbolic components}\label{ss: renormalizable parameters on the boundaries of hyperbolic components}
In this section, we discuss parameters on the boundaries of hyperbolic components, which are renormalizable. 

Let $\mathcal{U}$ be a component of $\hat{\mathcal{H}}$, $v_0 \in \partial \mathcal{U}$. By Lemma \ref{lm: non-degenerated annuli}, there is $p \ge 1$ and a sequence of puzzle pieces $\{\hat{P}^{v_0}_{n_0+kp}\}_{k \ge 0}$ such that $f^p_{v_0}: \hat{P}^{v_0}_{n_0+(k+1)p} \to \hat{P}^{v_0}_{n_0+kp}$ is a quadratic-like map, and the small filled-in Julia set $K_{v_0}=\bigcap_{k \ge 0} \overline{\hat{P}^{v_0}_{n_0+kp}}$ is connected. From the procedure in \S \ref{ss: transfer from the dynamical plane to the parameter plane}, there are para-puzzle pieces $\{\hat{\mathcal{P}}_{n_0+kp}\}_{k \ge 0}$ containing $v_0$. 

Recall that $\hat{P}^{v_0}_{n}$ admits a holomorphic motion $\tilde{H}_n$ parameterized by $v \in \hat{P}_{n}$. However, for distinct $n_1 \le n_2$, $f^{n_2-n_1}_v$ may not map $\hat{P}^v_{n_2}$ to $\hat{P}^v_{n_1}$, unless $v=v_0$. Therefore, we have to modify $\hat{P}^v_n$ again such that there is dynamical relation between puzzle pieces of different depths. 

Starting with $n_0$, let $\tilde{P}^v_{n_0}=\hat{P}^{v}_{n_0}$. Since the holomorphic motion $\tilde{H}^v_{n_0}$ agrees with $H^v_n$ on the graph $I^v_{n_0+1}$, the components of $f^{-p}_v(\tilde{P}^v_{n_0})$ are contained in the unbounded puzzle pieces of depth $n_1=n_0+p$. Let $\tilde{P}^v_{n_1}$ be the component of $f^{-p}_v(\tilde{P}^v_{n_0})$ contained in $P^v_{n_1}$. Since every puzzle piece contains at most one critical point of $f_v$, by Riemann-Hurwitz formula, $\tilde{P}^v_{n_1}$ is simply connected. 

$\{v \in \hat{\mathcal{P}}_{n_1}: -2v \in \tilde{P}^v_{n_1}\}$ is non-empty since it contains $v_0$. Define $\tilde{\mathcal{P}}_{n_1}$ to be the component of $\{v \in \hat{\mathcal{P}}_{n_1}: -2v \in \tilde{P}^v_{n_1}\}$ containing $v_0$. Inductively, for $n_k=n_0+kp$, $k \in \mathbb{N}$, one can define $\tilde{P}^v_{n_k}$ and $\tilde{\mathcal{P}}_{n_k}$ such that for $v \in \tilde{\mathcal{P}}_{n_k}$, $f^p_v: \tilde{P}^v_{n_k} \to \tilde{P}^v_{n_{k-1}}$ is a polynomial-like mapping. Let $\mathcal{M}_{v_0}=\bigcap_{k \ge 0} \overline{\tilde{\mathcal{P}}_{n_k}}$, which is a compact set containing $v_0$. 

\begin{definition}\label{def: copy of Mandelbrot set}
    A set $\mathcal{M}_0$ in the parameter plane of $f_v$ is called a copy of the Mandelbrot set $\mathcal{M}$, if there is a homeomorphism $\chi: \mathcal{M}_0 \to \mathcal{M}$ and an integer $p \ge 1$ satisfying 
    \begin{enumerate}
        \item $\mathcal{M}_0$ consists of parameters whose critical orbits do not escape to $\infty$;
        \item for $v \in \mathcal{M}_0$, $f^p_v$ is a renormalization of $f_v$;
        \item $p_v(z)=z^2+\chi(v)$ is conjugate to $f^p_v$ on a neighobrhood of the filled-in Julia set of $p_{v}$. 
    \end{enumerate}
\end{definition}

\begin{lemma}\label{lm: copy of Mandelbrot set}
    Let $\mathcal{U}$ be a component of $\hat{\mathcal{H}}$, $v_0 \in \partial \mathcal{U}$ be a renormalizable parameter. Then $\mathcal{M}_{v_0}$ is a copy of the Mandelbrot set. 
\end{lemma}
\begin{proof}
    By Lemma \ref{lm: non-degenerated annuli}, there is $n_0 \ge 1$, $p \ge 1$, such that for $v_0 \in \partial \mathcal{U}$, $k \ge 1$, $\hat{P}^{v_0}_{n_k} \to \hat{P}^{v_0}_{n_{k-1}}$ is a renormalization of $f_{v_0}$. 

    Let $\tilde{\mathcal{P}}_{n_k}$ be as above. Consider the family $\mathscr{F}=\{f^p_v: \tilde{P}^v_{n_1} \to \tilde{P}^v_{n_0}: v \in \tilde{\mathcal{P}}_{n_1}\}$. We will prove that $\mathscr{F}$ is an analytic family of polynomial-like mappings in the sense of \cite{DH85}, Chapter II.1. 

    Indeed, let $\mathcal{W}=\{(v,z): v \in \tilde{\mathcal{P}}_{n_1}, z \in \tilde{P}^v_{n_0}\}$, $\mathcal{W}'=\{(v,z): v \in \tilde{\mathcal{P}}_{n_1}, z \in \tilde{P}^v_{n_1}\}$, $F(v,z)=(v,f^p_v(z))$. Then 
    \begin{enumerate}
        \item $\mathcal{W}'$ and $\mathcal{W}$ are homeomorphic to $\tilde{\mathcal{P}}_{n_1} \times \mathbb{D}$, since $\tilde{P}^{v}_{n_k}$ are all simply connected;
        \item $\overline{\mathcal{W}'} \cap \mathcal{W}=\{(v,z): v \in \tilde{\mathcal{P}}_{n_1}, z \in \overline{\tilde{P}^v_{n_1}}\}$, and the projection $\overline{\mathcal{W}'} \cap \mathcal{W} \to \tilde{\mathcal{P}}_{n_1}$ is proper;
        \item $F: \mathcal{W}' \to \mathcal{W}$ is analytic and proper.
    \end{enumerate} 

    Since $\tilde{\mathcal{P}}_{n_1}$ is connected and $f^p_{v_0}: \tilde{P}^{v_0}_{n_1} \to \tilde{P}^{v_0}_{n_0}$ is quadratic-like, $\mathscr{F}$ forms a quadratic-like family. Let $\mathcal{M}_{F}=\{v \in \tilde{\mathcal{P}}_{n_1}: K_v \text{ is connected}\}$, where $K_v=\bigcap_{k \ge 0} \overline{\tilde{P}^v_{n_k}}$ is the small filled-in Julia set of $f^p_v: \tilde{P}^v_{n_1} \to \tilde{P}^v_{n_0}$. We claim that $\mathcal{M}_F=\mathcal{M}_{v_0}$. Indeed, for $v \in \mathcal{M}_F$, $K_v$ is connected thus the orbit of $-2v$ never leaves $\tilde{P}^v_{n_0}$. Thus $-2v \in \tilde{P}^v_{n_k}$ for all $k \ge 0$ and $v \in \mathcal{M}_{v_0}$. Conversely, for $v \in \mathcal{M}_{v_0}$, $-2v \in \tilde{P}^v_{n_k}$ for all $k \ge 0$. Thus the orbit of $-2v$ remains in $\tilde{P}^v_{n_0}$ and $K_v$ is connected. So $v \in \mathcal{M}_F$. 

    For the quadratic-like mapping $f^p_v: \tilde{P}^v_{n_1} \to \tilde{P}^v_{n_0}$, the Straightening Theorem (\cite{DH85}, Chapter I, Theorem 1) gives a continuous map $\chi: \tilde{P}_{n_1} \to \mathbb{C}$, and $\chi(\mathcal{M}_{v_0})=\mathcal{M}$, such that $f^p_v: \tilde{P}^v_{n_1} \to \tilde{P}^v_{n_0}$ is quasiconformally conjugate to the quadratic polynomial $z \mapsto z^2+\chi(v)$ on a neighborhood of $K_v$. 
    
    By \cite{DH85}, Theorem 4, to show that $\chi: \mathcal{M}_{v_0} \to \mathcal{M}$ is a homeomorphism, it would be sufficient to show that there is a unique $v_1 \in \mathcal{M}_{v_0}$ such that $f^p_{v_1}(c_{v_1})=c_{v_1}$, where $c_v$ is the critical point of $f^p_v: \tilde{P}^v_{n_{k+1}} \to \tilde{P}^v_{n_k}$. Consider a curve $\gamma=\partial \tilde{\mathcal{P}}_{n_k}$ surrounding $\mathcal{M}_{v_0}$. Then it would be sufficient to show that $f^p_v(c_v)-c_v$ turns around $0$ once as $v$ goes on $\gamma$. 
    
    Since $f^p_v: \tilde{P}^v_{n_{k+1}} \to \tilde{P}^v_{n_k}$ has only one critical value $-2v$, $f^p_v(c_v)=-2v$. Let $h_{n_k-1}$ be the transfer mapping in Lemma \ref{lm: transfer from dynamical plane to parameter plane}, which is well-defined on $\partial \mathcal{P}_{n_k-1} \cap \mathcal{I}_{n_k}$ and can be extended to $\gamma$ as in the proof of Lemma \ref{lm: non-degenerated para-annuli}. Let $G_0(v)=h_{n_k-1}(v)-c_{v_0}$, $G_1(v)=-2v-c_v$. Indeed, let $\delta: [0,1] \to \overline{\tilde{\mathcal{P}}_{n_k}}$ be a Jordan arc connecting $v_0$ and $v$. Let $G(t,v)=H_{n_k-1}(\delta(t), h_{n_k-1}(v))-c_{\delta(t)}$. Then $G(0,v)=H_{n_k-1}(v_0, h_{n_k-1}(v))-c_{v_0}=h_{n_k-1}(v)-c_{v_0}$, $G(1,v)=H_{n_k-1}(v, h_{n_k-1}(v))-c_v=-2v-c_v$. Thus $G_0$ and $G_1$ are homotopic in $\overline{\tilde{\mathcal{P}}_{n_k}}$. 
    
    To show that $G_1(\gamma(t))$ turns around $0$ exactly once, it remains to show that $G_0(\gamma(t))$ turns around $0$ exactly once. For $v \in \partial \tilde{\mathcal{P}}_{n_k}$, $h_{n_k-1}(v) \in \partial \tilde{P}^{v_0}_{n_k} =\partial \hat{P}^{v_0}_{n_k}$, and $\hat{P}^{v_0}_{n_k}$ is a puzzle piece surrounding $c_{v_0}$. Hence $G_0(v)$ turns around $0$ at least once as $v$ goes along $\gamma$. Since $h_{n_k-1}$ is a homeomorphism by Lemma \ref{lm: transfer from dynamical plane to parameter plane}, $G_0(v)$ turns around $0$ exactly once.  
\end{proof}

\begin{lemma}\label{lm: attracting periodic points give renormalization}
    Suppose that $f_v$ has an attracting or a parabolic periodic point $x \neq 0$, then $f_v$ is renormalizable and $v$ is contained in a copy of $\mathcal{M}$.
\end{lemma}
\begin{proof}
    Since $f_v$ has an attracting or a parabolic periodic point $x \neq 0$, $f_v$ is not in any component of $\hat{\mathcal{H}}$. Moreover, the orbit of $-2v$ is bounded since it is attracted by this attracting cycle. Therefore, the puzzle pieces $P^v_n(x)$, $n \ge 0$, which contain $x$, are well-defined. Suppose that $x$ is of period $p \ge 1$. Then $f^p_v(P^v_{n+p}(x))=P^v_n(x)$, $n \ge 0$. 
    
    We claim that $f^p_v: \hat{P}^v_{n+p}(x) \to \hat{P}^v_n(x)$ is a quadratic-like mapping. Indeed, if $\hat{P}^v_{n+p}(x)$ contains no critical points of $f^p_v$, then there is a branch $g$ of $f^{-p}_v$ on $\hat{P}^v_n(x)$, which fixes $x$. In other words, $g: \hat{P}^v_n(x) \to \hat{P}^v_{n+p}(x)$ is a conformal isomorphism. By Schwarz' Lemma, $|g'(x)|<1$. This implies that $|(f^p_v)'(x)|>1$, a contradiction. Thus $f^p_v: \hat{P}^v_{n+p}(x) \to \hat{P}^v_n(x)$ contains a critical point of $f^p_v$ and $-2v \in \hat{P}^v_{n+j}(f^j_v(x))$ for some $0 \le j \le p-1$. Since $\hat{P}^v_{n+j}(f^j_v(x))$, $0 \le j \le p-1$ are mutually disjoint, there is exactly one puzzle piece $\hat{P}^v_{n+j}(f^j_v(x))$ containing $-2v$, and $f^p_v: \hat{P}^v_{n+p}(x) \to \hat{P}^v_n(x)$ is of degree $2$. 

    Since $x$ is attracting or parabolic, the orbit of $-2v$ remains in $\hat{P}^v_{n+j}(f^j_v(x))$. This yields that the small filled-in Julia set of the quadratic-like mapping $f^p_v: \hat{P}^v_{n+p}(x) \to \hat{P}^v_n(x)$ is connected so that $f^p_v: \hat{P}^v_{n+p}(x) \to \hat{P}^v_n(x)$ is a renormalization of $f_v$. By Lemma \ref{lm: copy of Mandelbrot set}, $v$ belongs to a copy of $\mathcal{M}$. 
\end{proof}

Now the boundedness of hyperbolic components of type D follows from Lemma \ref{lm: attracting periodic points give renormalization}.
\begin{proof}[Proof of Lemma \ref{lm: boundedness of hyperbolic components of type D}]
    For a hyperbolic component $\mathcal{D}$ of type D, $v_0 \in \mathcal{D}$, $f_{v_0}$ is renormalizable by Lemma \ref{lm: attracting periodic points give renormalization}, and $v_0$ lies in a copy $\mathcal{M}_{v_0}$ of $\mathcal{M}$ by Lemma \ref{lm: copy of Mandelbrot set}. Moreover, $\mathcal{M}_{v_0} \subset \tilde{\mathcal{P}}_{n_0}$ for some $n_0 \ge 1$, where $\tilde{\mathcal{P}}_{n_0}$ is bounded. As a result, $\mathcal{D}$ is entirely contained in $\mathcal{M}_{v_0}$ hence bounded. 
\end{proof}

\subsection{Local connectivity of $\partial \mathcal{H}_0$}\label{ss: local connectivity of the boundary of H_0}
In this section, we consider the boundary of $\mathcal{H}_0$. 
\begin{proposition}\label{prop: renormalizable parameters on the boundary of H_0}
    Let $v_0 \in \partial \mathcal{H}_0$ such that $f_{v_0}$ is renormalizable. Then $\partial \mathcal{H}_0$ is locally connected at $v_0$. More precisely, there are exactly two parameter rays $G_{\underline{s}_+}$, $G_{\underline{s}_-}$ and a parameter internal ray $\mathcal{R}_0(\theta)$, which all land at $v_0$. And $\overline{G_{\underline{s}_+} \cup G_{\underline{s}_-}}$ seperates $\mathcal{M}_{v_0} \setminus \{v_0\}$ from $\partial \mathcal{H}_0$. In the dynamical plane of $f_{v_0}$, the dynamic rays $g^{v_0}_{\underline{s}_{\pm}}$ and the internal ray $R^{v_0}_{0}(\theta)$ land at a parabolic periodic point $x_{v_0} \in \partial B_{v_0}$, and $\overline{g^{v_0}_{\underline{s}_+} \cup g^{v_0}_{\underline{s}_-}}$ seperates $-2v_0$ from $B_{v_0}$. 
\end{proposition}
\begin{proof}
    By Lemma \ref{lm: non-degenerated annuli}, there exists $p \ge 1$, a sequence of puzzle pieces $\{\hat{P}^{v_0}_{n_k}\}$ containing $-2v_0$, $k \ge 0$, $n_k=n_0+p$, such that $f^p_{v_0}: \hat{P}^{v_0}_{n_{k+1}} \to \hat{P}^{v_0}_{n_k}$ is a renormalization of $f_{v_0}$. 
    
    $\mathcal{M}_{v_0}$ is contained in every $\mathcal{P}_{n_k}$, $k \ge 0$. $\partial \mathcal{P}_{n_k} \cap \mathcal{H}_0$ consists of two parameter internal rays $\mathcal{R}_0(\theta^{\pm}_{k})$, which land at $v^{\pm}_{k}$ respectively. And $v^{\pm}_k$ are the landing points of parameter rays $G_{\underline{s}^{\pm}_k}$ respectively. By the construction of $\mathcal{P}_{n_k}$, $|\theta^+_{k+1}-\theta^-_{k+1}|=|\theta^+_k-\theta^-_k|/2^p$. Thus there is $\theta$ such that $\theta^{\pm}_k \to \theta$ as $k \to \infty$. We see that $\theta$ is of period $p$ under the angle doubling map. Similar to the proof of \cite{QW25}, Lemma 4.5, $\underline{s}^{\pm}_k$ converges to $\underline{s}_{\pm}$ as $k \to \infty$ respectively, and $\underline{s}_{\pm}$ are also of period $p$ under the shift.  

    By Lemma \ref{lm: copy of Mandelbrot set}, there is a homeomorphism $\chi_{v_0}: \mathcal{M}_{v_0} \to \mathcal{M}$. We will show that the rays $G_{\underline{s}_{\pm}}$, $\mathcal{R}_0(\theta)$ all land at $\chi^{-1}_{v_0}(1/4)$. Since $\theta$ is periodic under the angle doubling map, by Proposition \ref{prop: landing theorem for parameter internal rays}, $\mathcal{R}_0(\theta)$ lands at $v_1$ such that $R^{v_1}_0(\theta)$ lands either at a parabolic periodic point $x_{v_1}$ of $f_{v_1}$ or at $-2v_1 \in \partial B_{v_1}$. Since $\theta$ is periodic, $-2v_1$ will be fixed by $f^p_{v_1}$ if $R^{v_1}_0(\theta)$ lands at $-2v_1$. However, this implies that $f^{p-1}_{v_1}(-2v_1)$ is a critical point $2k\pi$ for some $k \in \mathbb{Z} \setminus \{0\}$ and $2k\pi$ is a superattracting periodic point of $f_{v_1}$, contradicting that $v_1 \in \partial \mathcal{H}_0$. Thus $R^{v_1}_0(\theta)$ lands at a parabolic periodic point $x_{v_1}$. Since $\theta$ is in every interval $(\theta^-_k, \theta^+_k)$, $v_1 \in \mathcal{M}_{v_0}$. The multiplier of $x_{v_1}$ is $1$ since $R^{v_1}_0(\theta)$ is fixed by $f^p_{v_1}$. Thus $v_1=\chi^{-1}_{v_0}(1/4)$.

    Similarly, by Proposition \ref{prop: landing theorem for parameter rays}, $G_{\underline{s}_+}$ lands at some $v_2$ , such that $g^{v_2}_{\underline{s}_+}$ lands at a parabolic periodic point $x_{v_2}$. Recall that $\tilde{\mathcal{P}}_{n_k}=\{v \in \hat{\mathcal{P}}_{n_k}: -2v \in \tilde{P}^v_{n_k}\}$. Therefore, for $v \in G_{\underline{s}^{\pm}_{k}}$, $-2v \in g^{v}_{\underline{s}^{\pm}_k} \subset \partial P^v_{\underline{s}^{\pm}_k}$. Thus the landing points $v^{\pm}_k$ of $G_{\underline{s}^{\pm}_k}$ belong to $\partial \tilde{\mathcal{P}}_{n_l}$, $0 \le l \le k$. Since $v_2$ is the limit of $v^+_k$, $v_2 \in \mathcal{M}_{v_0}$. Again, since $g^{v_2}_{\underline{s}_+}$ is fixed by $f^p_{v_2}$, the multiplier of $x_{v_2}$ is $1$, i.e. $v_2=\chi^{-1}_{v_0}(1/4)=v_1$. The same argument holds for the parameter ray $G_{\underline{s}_-}$. As a result, the rays $G_{\underline{s}_{\pm}}$, and $\mathcal{R}_0(\theta)$ all land at $\chi^{-1}_{v_0}(1/4)$. 

    Since $\mathcal{M}_{v_0}$ is contained in the closure of the component of $\mathbb{C} \setminus \overline{G_{\underline{s}_+} \cup G_{\underline{s}_-}}$, the only possible intersecting point of $\mathcal{M}_{v_0}$ and $\partial \mathcal{H}_0$ is $v_1$. Thus $v_1=v_0$, and $\partial \mathcal{H}_0$ is locally connected at $v_0$. 

    The results in the dynamical plane of $f_{v_0}$ follow from \cite{QW25}, Lemma 4.3.
\end{proof}

From the proof of Proposition \ref{prop: renormalizable parameters on the boundary of H_0}, we obtain the following
\begin{corollary}\label{cor: every renormalizable parameter on the boundary of H_0 is a cusp}
    Let $v_0 \in \partial \mathcal{H}_0$ be renormalizable, $\mathcal{M}_{v_0}$ be the copy of $\mathcal{M}$ containing $v_0$, and $\chi_{v_0}: \mathcal{M}_{v_0} \to \mathcal{M}$ be the homeomorphism. Then $v_0=\chi^{-1}_{v_0}(1/4)$, which is a cusp of $\mathcal{M}_{v_0}$. 
\end{corollary}

\begin{corollary}\label{cor: H_0 is a Jordan domain}
    $\mathcal{H}_0 \cup \{0\}$ is a Jordan domain.
\end{corollary}
\begin{proof}
    By Lemma \ref{lm: hyperbolic components are simply connected}, $\mathcal{H}_0 \cup \{0\}$ is simply connected. By Corollary \ref{cor: local connectivity at non-renormalizable parameters} and Proposition \ref{prop: renormalizable parameters on the boundary of H_0}, $\partial \mathcal{H}_0 \cup \{0\}$ is locally connected. The Maximum Modulus Principle implies that it is a Jordan curve. 
\end{proof}

\begin{definition}\label{def: wake}
    Let $v_0 \in \partial \mathcal{H}_0$ be renormalizable. Let $\mathcal{R}_0(\theta)$, $G_{\underline{s}_{\pm}}$ be the rays in Proposition \ref{prop: renormalizable parameters on the boundary of H_0}. The component of $\mathbb{C} \setminus \overline{G_{\underline{s}_+} \cup \underline{s}_-}$ not containing $\mathcal{H}_0$ is called the wake of $v_0$, denoted by $W(v_0)$. 
\end{definition}

By Proposition \ref{prop: renormalizable parameters on the boundary of H_0}, $\mathcal{M}_{v_0} \setminus \{v_0\} \subset W(v_0)$. 

\begin{lemma}\label{lm: common landing point for parameter in a wake}
    For $v \in W(v_0)$, the rays $R^{v}_0(\theta)$ and $g^v_{\underline{s}_{\pm}}$ land at the same point. 
\end{lemma}
\begin{proof}
    By Lemma \ref{lm: holomorphic motion of para-graphs}, $\partial P^v_{n}$ moves holomorphically for $v \in \mathcal{M}_{v_0}$. Passing to the limit, $\overline{R^v_0(\theta) \cup g^v_{\underline{s}_+} \cup g^v_{\underline{s}_-}}$ moves holomorphically for $v \in \mathcal{M}_{v_0}$. Thus the rays $R^{v}_0(\theta)$ and $g^v_{\underline{s}_{\pm}}$ land at the same point for $v \in \mathcal{M}_{v_0}$. 
    
    For $v \in W(v_0)$, the critical value $-2v$ never meets the rays $R^{v}_0(\theta)$ and $g^v_{\underline{s}_{\pm}}$. Thus $\overline{R^v_0(\theta) \cup g^v_{\underline{s}_+} \cup g^v_{\underline{s}_-}}$ moves holomorphically for $v \in W(v_0)$. Therefore, the rays $R^{v}_0(\theta)$ and $g^v_{\underline{s}_{\pm}}$ land at the same point for $v \in W(v_0)$.
\end{proof}

At the end of this section, we can describe the boundary of $\mathcal{H}_0$. 

\begin{proposition}\label{prop: dichotomy of parameters on the boundary of H_0}
    Let $v_0 \in \partial \mathcal{H}_0$ and $v_0 \neq 0$. Then there is a unique parameter internal ray $\mathcal{R}_0(\theta)$ landing at $v_0$. Moreover, the following dichotomy holds.
    \begin{enumerate}
        \item either $f_{v_0}$ is not renormalizable, and it is the landing point of a unique parameter ray $G_{\underline{s}}$. In this case, in the dynamical plane of $f_{v_0}$, $R^{v_0}_0(\theta)$ and $g^{v_0}_{\underline{s}}$ both land at $-2v_0$, and they are the only internal ray and dynamic ray landing at $v_0$;
        \item or $f_{v_0}$ is renormalizable, and there are exactly two parameter rays $G_{\underline{s}_{\pm}}$ landing at $v_0$. In this case, $v_0$ is the cusp of $\mathcal{M}_{v_0}$. In the dynamical plane of $f_{v_0}$, $R^{v_0}_{\theta}$ and $g^{v_{0}}_{\underline{s}_{\pm}}$ all land at a parabolic periodic point $x_{v_0}$ of $f_{v_0}$, and no other internal rays or dynamic rays landing at $x_{v_{0}}$.
    \end{enumerate}
\end{proposition}
\begin{proof}
    By Proposition \ref{prop: Phi_0 is a covering map of degree 2}, $\Phi_0$ maps both components of $\mathcal{H}_0 \setminus \mathbb{R}$ to $\mathbb{D}\setminus [0,1)$ homeomorphically. Thus every $v \in \partial \mathcal{H}_0 \setminus \mathbb{R}$ is the landing point of exactly one parameter internal ray $\mathcal{R}_0(\theta)$. For $v \in \partial \mathcal{H}_0 \cap \mathbb{R}$, $v$ is the landing point of $\mathcal{R}_0(0)$. It cannot be the landing point of other parameter internal rays since $\mathcal{R}_0(\theta)$ lands at $\partial \mathcal{H}_0 \setminus \mathbb{R}$ if $\theta \neq 0$. 

    Now let $v_0 \in \partial \mathcal{H}_0 \setminus \{0\}$. There are the following two cases. 

    \textbf{Case 1.} $f_{v_0}$ is not renormalizable. By Lemma \ref{lm: non-degenerated annuli}, there is a sequence $\{\hat{P}^{v_0}_{n_k}\}_{k \ge 0}$ of puzzle pieces such that $\bigcap_{k \ge 0} \overline{\hat{P}^{v_0}_{n_k}}=\{-2v_0\}$, and $\partial \mathcal{P}_{n_k}$ and $\partial P^{v_0}_{n_k}$ are homeomorphic by Lemma \ref{lm: transfer from dynamical plane to parameter plane}. Since $\mathcal{R}_0(\theta)$ lands at $v_0$, it enters all $\mathcal{P}_{n_k}$. Thus $R^{v_0}_0(\theta)$ enters all $P^{v_0}_{n_k}$. Since $\bigcap_{k \ge 0} \overline{\hat{P}^{v_0}_{n_k}}=\{-2v_0\}$, $R^{v_0}_0(\theta)$ lands at $-2v_0$. 

    Since $-2v_0 \in \partial B_{v_0}$, the curve $\Gamma=[-2v,0] \cup \mathbb{R}_{\ge 0}$ only intersects $J(f_{v_0})$ at $(2k+1)\pi$, $k \ge 0$ and $-2v_0$. Hence every two half strips are seperated by the closure of countably many Fatou components, and every dynamic ray $g_{\underline{s}}$ is entirely contained in $P_{s_0}$ provided $\underline{s}=(s_0, s_1, \ldots)$. Therefore, every dynamic ray landing at $-2v_0$ has the same external address. This implies that there is exactly one dynamic ray $g^{v_0}_{\underline{s}}$ landing at $-2v_0$. So $G_{\underline{s}}$ is the unique parameter ray landing at $v_0$. 

    \textbf{Case 2.} $f_{v_0}$ is renormalizable. By Proposition \ref{prop: renormalizable parameters on the boundary of H_0}, $v_0$ is the cusp of $\mathcal{M}_{v_0}$ and is the common landing point of $\mathcal{R}_0(\theta)$ and $G_{\underline{s}_{\pm}}$. And in the dynamical plane of $f_{v_0}$, $R^{v_0}_0(\theta)$ and $g^{v_0}_{\underline{s}_{\pm}}$ all land at a common parabolic periodic point $x_{v_0}$. It remains to show that $G_{\underline{s}_{\pm}}$ are the only parameter rays landing at $v_0$. 

    Suppose by contradiction that $G_{\underline{s}'}$ also lands at $v_0$. Then it enters all $\tilde{\mathcal{P}}_{n_k}$. So we may assume that $\underline{s}_- \prec \underline{s}' \prec \underline{s}_+$, where the order of external addresses is defined in \cite{QW25}. Therefore, $\mathcal{M}_{v_0}$ is entirely contained in one of the components of $\mathbb{C} \setminus \overline{G_{\underline{s}_-} \cup G_{\underline{s}'} \cup G_{\underline{s}_+}}$. We may assume that $\mathcal{M}_{v_0} \setminus \{v_0\}$ and $G_{\underline{s}_+}$ are in the different component of $\mathbb{C} \setminus \overline{G_{\underline{s}_-} \cup G_{\underline{s}'}}$. Then $g^{v_0}_{\underline{s}'}$ also lands at $x_{v_0}$. Consider a repelling periodic point $K_{v_0}$ which is the landing point of a periodic dynamic ray $g^{v_0}_{\underline{t}}$ and $\underline{s}' \prec \underline{t} \prec \underline{s}_+$. Then $G_{\underline{t}}$ lands at a point $v' \in \mathcal{M}_{v_0}$ since it enters every $\tilde{\mathcal{P}_{n_k}}$. However, this is impossible since $\mathcal{M}_{v_0} \setminus \{v_0\}$ is seperated from $G_{\underline{s}_+}$ by $\overline{G_{\underline{s}_-} \cup G_{\underline{s}'}}$. 
\end{proof}

\subsection{Local connectivity of the boundaries of hyperbolic components of type C}\label{ss: local connectivity of the boundaries of hyperbolic components of type C}
In this section, we prove the local connectivity of the boundaries of hyperbolic components of type C. 
\begin{proposition}\label{prop: renormalizable parameters on the boundaries of components of type C}
    Let $\mathcal{U}$ be a component of $\mathcal{H}_m$, $m \ge 1$, $v_0 \in \partial \mathcal{U}$ be renormalizable, with period of renormalization $p$. Then $\partial \mathcal{U}$ is locally connected at $v_0$. More precisely, the following statements are true:
    \begin{enumerate}
        \item $\mathcal{M}_{v_0} \cap \partial \mathcal{H}_0=\chi^{-1}_{v_0}(1/4)=:v_1$, which is the cusp of $\mathcal{M}_{v_0}$;
        \item there are exactly two parameter rays $G_{\underline{s}^{\pm}}$ and a parameter internal ray $\mathcal{R}_{\mathcal{U}}(\theta)$ landing at $v_0$, where $\underline{t}^{\pm}=\sigma^m(\underline{s}^{\pm})$ and $\theta$ are of period $p$, and $\underline{s}^{\pm}$, $\theta$ are from Proposition \ref{prop: renormalizable parameters on the boundary of H_0};
        \item $m$ is a multiple of $p$;
        \item $\overline{G_{\underline{s}^+} \cup G_{\underline{s}^-}}$ seperates $\partial \mathcal{U}$ from $\mathcal{M}_{v_0} \setminus \{v_0\}$.
    \end{enumerate} 
\end{proposition}
\begin{proof}
    Let $\mathcal{U}$, $v_0$ be as above. Since $v_0$ is renormalizable, by Lemma \ref{lm: non-degenerated annuli}, there are puzzle pieces $\hat{P}^{v_0}_{n_k}$, $k \ge 0$, such that $f^p_{v_0}: \hat{P}^{v_0}_{n_{k+1}} \to \hat{P}^{v_0}_{n_k}$ is a quadratic-like mapping. And by Lemma \ref{lm: copy of Mandelbrot set}, there are para-puzzle pieces $\tilde{\mathcal{P}}_{n_k}$, $k \ge 0$ such that $\mathcal{M}_{v_0}= \bigcap_{k \ge 0} \overline{\tilde{\mathcal{P}}_{n_k}}$ is a copy of $\mathcal{M}$. By the definition of para-puzzle pieces, $\partial \mathcal{P}_{n} \cap \mathcal{U}$ consists of twp parameter internal rays $\mathcal{R}_{\mathcal{U}}(\theta^{\pm}_n)$, landing at $v^{\pm}_n$ respectively. And $v^{\pm}_n$ are the landing points of parameter rays $G_{\underline{s}^{\pm}_n}$ respectively, which are also on $\partial \mathcal{P}_{n}$. Then the dynamic rays $g^{v_0}_{\underline{s}^{\pm}_n}$ are on $\partial P^{v_0}_{n}$, landing at the same point with some preimages of $f^{-m}_{v_0}(R^{v_0}_0(\theta^{\pm}_n))$. 

    First we show that $\tilde{\mathcal{P}}_{n_k} \cap \mathcal{H}_0 \neq \varnothing$ so that $\mathcal{M}_{v_0} \cap \partial \mathcal{H}_0 \neq \varnothing$. Indeed, consider $P^{v_0}_{l}$, $l \ge m+1$. Since $\mathcal{U}$ is of depth $m$, $f^m_{v_0}(P^{v_0}_l)$ is a puzzle piece intersecting $B_{v_0}$. Thus for $kp \ge m$, $n_l \ge m$, $P^{v_0}_{n_l}=f^{kp}_{v_0}(P^{v_0}_{n_{l+k}})$ is a puzzle piece intersecting $B_{v_0}$. By the homeomorphism in Proposition \ref{lm: transfer from dynamical plane to parameter plane}, $\tilde{P}_{n_k} \cap \mathcal{H}_0 \neq \varnothing$. Therefore, $\mathcal{P}_{n_k} \cap \mathcal{H}_0$ contains two parameter internal rays $\mathcal{R}_0(\eta^{\pm}{n_k})$, which land at the same points with parameter rays $G_{\underline{t}^{\pm}_{n_k}}$ respectively. 

    Then we show that $\mathcal{M}_{v_0} \cap \partial \mathcal{H}_0$ at a single point $v_1=\chi^{-1}_{v_0}(1/4)$. Suppose that $\mathcal{M}_{v_0}$ intersects $\partial \mathcal{H}_0$ at two points $v_1$, $v_2$. Then by Proposition \ref{prop: renormalizable parameters on the boundary of H_0}, $v_i=\chi^{-1}_{v_i}(1/4)$, $i=1,2$. Since $v_1, v_2 \in \mathcal{M}_{v_0}$, $\chi_{v_1}=\chi_{v_0}=\chi_{v_2}$. Since $\chi_{v_0}$ is a homeomorphism, $v_1=v_2$. 
    
    Sicne the image of $R^{v_0}_0(\eta^{\pm}_{n_k})$ under $f^m_{v_0}$ are contained in $\partial P^{v_0}_{n_k-m}(f^m_{v_0}(-2v_0)) \cap B_{v_0}$, which consists of $R^{v_0}_0(\theta^{\pm}_{n_k-m})$. Since $f_{v_0}|_{B_{v_0}}$ doubles the angles of internal rays in $B_{v_0}$, $2^m\eta^{\pm}_{n_k}=\theta^{\pm}_{n_k-m}$. Similarly, $f^m_{v_0}$ maps $g^{v_{0}}_{\underline{s}^{\pm}_{n_k}}$ and $g^{v_0}_{\underline{t}^{\pm}_{n_k}}$ to the dynamic rays landing at the same points as $R^{v_0}_0(\theta^{\pm}_{n_k-m})$. Hence $\sigma^m(\underline{t}^{\pm}_{n_k})=\sigma^m(\underline{s}^{\pm}_{n_k})$. 
    
    We show that $P^{v_0}_{n_k}(f^m_{v_0}(-2v_0))=P^{v_0}_{n_k}$ so that $m$ is a multiple of $p$. Since $\mathcal{U}$ is of depth $m$, the component $U$ of $f^{-m}_{v_0}(B_{v_0})$, on whose boundary $g^{v_0}_{\underline{s}^{\pm}_{n_k+m}}$ land, satisfyies $f^m_{v_0}(U)=B_{v_0}$, and $f^j_{v_0}(U) \neq B_{v_0}$, $0 \le j \le m-1$. As a result, $\sigma^j(\underline{s}^{\pm}_{n_k+m})$ never land on $\partial B_{v_0}$. Hence $f^{m-1}_{v_0}: \hat{P}^{v_0}_{n_k+m} \to \hat{P}^{v_0}_{n_k+1}(f^{m-1}_{v_0}(-2v_0))$ is a conformal isomorphism. $f_{v_0}: \hat{P}^{v_0}_{n_k+1}(f^{m-1}_{v_0}(-2v_0)) \to \hat{P}^{v_0}_{n_k}(f^m_{v_0}(-2v_0))$ is a branched covering of degree $2$ since it is not a homeomorphism on $\partial \hat{P}^{v_0}_{n_k+1}(f^{m-1}_{v_0}(-2v_0))$ and every puzzle piece contains at most one critical point of $f_{v_0}$. This implies that $\hat{P}^{v_0}_{n_k+1}(f^{m-1}_{v_0}(-2v_0))$ contains a critical point of $f_{v_0}$ and $-2v_0 \in \hat{P}^{v_0}_{n_k}(f^m_{v_0}(-2v_0))$. Thus $P^{v_0}_{n_k}(f^m_{v_0}(-2v_0))=P^{v_0}_{n_k}$ and $m$ is a multiple of $p$. Meanwhile, this also gives that $\eta_{n_k}=\theta_{n_k}$ and $\underline{t}^{\pm}_{n_k}=\sigma^m(\underline{t}^{\pm}_{n_k+m})=\sigma^m(\underline{s}^{\pm}_{n_k+m})$. 

    Again, using the homeomorphism in Lemma \ref{lm: transfer from dynamical plane to parameter plane}, $\mathcal{R}_0(\theta^{pm}_{n_k})$ and $G_{\underline{t}^{\pm}_{n_k}}$ are on $\partial \mathcal{P}_{n_k}$. Since $\mathcal{P}_{n_k}(v_1)=\mathcal{P}_{n_k}$, by Proposition \ref{prop: renormalizable parameters on the boundary of H_0}, there is an angle $\theta$ and two external addresses $\underline{t}^{\pm}$, such that $\theta^{\pm}_{n_k} \to \theta$ and $\underline{t}^{\pm}_{n_k} \to \underline{t}^{\pm}$ as $k \to \infty$. Moreover, $\theta$ is of period $p$ under the angle doubling map and $\underline{t}^{\pm}$ are of period $p$ under the one-side shift $\sigma$. And $\mathcal{R}_0(\theta)$, $G_{\underline{t}^{\pm}}$ all land at $v_1$. 

    Consider the dynamic rays $g^{v_0}_{\underline{s}^{\pm}_{n_k}}$ landing on $\partial U$. They form two monotone sequences of external addresses. Thus there are $\underline{s}^{\pm}$ to which $\underline{s}^{\pm}_{n_k}$ converge as $k \to \infty$ respectively. 

    Finally, we show that $\mathcal{R}_{\mathcal{U}}(\theta)$ and $G_{\underline{s}^{\pm}}$ all land at $v_0$, and $\overline{G_{\underline{s}^+} \cup G_{\underline{s}^-}}$ seperates $\partial \mathcal{U}$ from $\mathcal{M}_{v_0} \setminus \{v_0\}$. Since $\theta$ is rational, $\mathcal{R}_{\mathcal{U}}(\theta)$ lands at some point $v_{\theta} \in \partial \mathcal{U}$. Similarly, $G_{\underline{s}^{\pm}}$ land at $v_{\pm} \in \partial \mathcal{U}$ respectively. 
    
    Since $\underline{s}^{\pm}$ are strictly preperiodic, $v_{\pm}$ are post-critically finite by Proposition \ref{prop: landing theorem for parameter rays}. Consider the dynamical plane of $f_{v_+}$. $f^m_{v_+}(-2v_+)$ is a fixed point of $f^p_{v_+}$, which is the landing point of $R^{v_+}_0(\theta)$ and $g^{v_+}_{\underline{t}^{\pm}}$ since $v_+ \in W(v_1)$ and $R^{v_1}_0(\theta)$ and $g^{v_1}_{\underline{t}^{\pm}}$ land at the same point by Proposition \ref{prop: renormalizable parameters on the boundary of H_0}. Pulling back to $-2v_+$ by $f^m_{v_+}$, $g^{v_+}_{\underline{s}^+}$, $g^{v_+}_{\underline{s}^{\pm}}$ both land at $-2v_+$. By Proposition \ref{prop: converse landing theorem for eventually periodic parameters}, $v_+$ is the landing point of $G_{\underline{s}^{\pm}}$, i.e. $v_+=v_-$. 

    Then we consider the parameter $v_{\theta}$. It is either post-critically finite or parabolic by Proposition \ref{prop: landing theorem for parameter internal rays}. By Proposition \ref{prop: converse landing theorem for eventually periodic parameters} and \ref{prop: converse landing theorem for parabolic parameters}, there is at least one parameter ray $G_{\underline{s}'}$ landing at $v_{\theta}$. If $v_{\theta} \neq v_+$, we may assume that $\underline{s}^- \prec \underline{s}'$. Since $v_{\theta}$ is the landing point of $\mathcal{R}_{\mathcal{U}}(\theta)$, it belongs to every para-puzzle piece $\tilde{\mathcal{P}}_{n_k}$. Thus $G_{\underline{s}'}$ enters every $\tilde{\mathcal{P}}_{n_k}$, and $\underline{s}^- \prec \underline{s}' \prec \underline{s}^-_{n_k}$ for every $k \ge 0$. Since $\underline{s}^-_{n_k} \to \underline{s}^-$ as $k \to \infty$, $\underline{s}'=\underline{s}^-$, contradicting that $v_{\theta} \neq v_+$. 
    
    Let $S$ be the component of $\mathbb{C} \setminus \overline{G_{\underline{s}^+ \cup G_{\underline{s}^-}}}$ containing $\mathcal{U}$. If $\mathcal{M}_{v_0} \cap S \neq \varnothing$, then we may choose a post-critically finite parameter $v' \in \mathcal{M}_{v_0} \cap S$. It is the landing point of some $G_{\underline{s}''}$. As above, $\underline{s}''$ must be equal to $\underline{s}^+$ or $\underline{s}^-$, contradicting that $v' \in S$ while $v_+ \in \partial S$. Thus $\mathcal{M}_{v_0} \cap \partial \mathcal{U}=\{v_+\}$. Since $v_0$ is also in $\mathcal{M}_{v_0} \cap \partial \mathcal{U}$, $v_0=v_+$. Therefore, $\overline{G_{\underline{s}^+} \cup G_{\underline{s}^-}}$ seperates $\partial \mathcal{U}$ from $\mathcal{M}_{v_0} \setminus \{v_0\}$, and $v_0$ is the landing point of $\mathcal{R}_{\mathcal{U}(\theta)}$ and $G_{\underline{s}^{\pm}}$.
\end{proof}

For the same reason as Corollary \ref{cor: H_0 is a Jordan domain}, hyperbolic components of type C are Jordan domains. 
\begin{corollary}\label{cor: hyperbolic components of type C are Jordan domains}
    Let $\mathcal{U}$ be a component of $\mathcal{H}_m$, $m \ge 1$. Then $\mathcal{U}$ is a Jordan domain.
\end{corollary}

Now Theorem \ref{thm: boundaries of hyperbolic components are Jordan curves} follows from Proposition \ref{prop: parameterization of hyperbolic components of type D}, Corollary \ref{cor: H_0 is a Jordan domain} and \ref{cor: hyperbolic components of type C are Jordan domains}.

\begin{corollary}\label{cor: dynamics of parameters on the boundaries of hyperbolic components of type C}
    Let $\mathcal{U}$ be a component of $\mathcal{H}_m$, $m \ge 1$, $v \in \partial \mathcal{U}$. Then $v$ is the landing point of a unique parameter internal ray $\mathcal{R}_{\mathcal{U}}(\theta)$, so that the internal ray $R^{v}_0(\theta)$ lands at $f^m_v(-2v) \in \partial B_{v}$. 
\end{corollary}
\begin{proof}
    By Corollary \ref{cor: hyperbolic components of type C are Jordan domains}, $\Phi_{\mathcal{U}}: \mathcal{U} \to \mathbb{D}$ extends to a homeomorphism $\Phi_{\mathcal{U}}: \overline{\mathcal{U}} \to \overline{\mathbb{D}}$. Thus every $v \in \partial \mathcal{U}$ is the landing point of a unique parameter internal ray $\mathcal{R}_{\mathcal{U}}(\theta)$. By the definition of $\mathcal{R}_{\mathcal{U}}(\theta)$, $f^m_v(-2v)$ is the landing point of $R^v_0(\theta)$. 
\end{proof}

\begin{corollary}\label{cor: at most one renormalizable parameter on the boundary of hyperbolic components of type C}
    Let $\mathcal{U}$ be a component of $\mathcal{H}_m$, $m \ge 1$. Then there is at most one $v_0 \in \partial \mathcal{U}$ which is renormalizable. And $v_0$ must be post-critically finite.
\end{corollary}
\begin{proof}
    Suppose by contradiction that $v_1 \neq v_2 \in \partial \mathcal{U}$ are both renormalizable. Then there is a Jordan curve $\gamma$ in $\mathcal{H}_0 \cup \mathcal{M}_{v_1} \cup \mathcal{M}_{v_2} \cup \mathcal{U}$ which bounds a bounded domain $\Omega$. Consider the two parameter rays $G_{\underline{s}^{\pm}}$ landing at $v_1$. The curve $\overline{G_{\underline{s}^+} \cup G_{\underline{s}^-}}$ seperates $\partial \mathcal{U}$ from $\mathcal{M}_{v_1} \setminus \{v_1\}$ by Proposition \ref{prop: renormalizable parameters on the boundaries of components of type C}. Thus one of $G_{\underline{s}^{\pm}}$ is contained in $\Omega$, which is a bounded domain. This is impossible. 

    From the proof of Proposition \ref{prop: renormalizable parameters on the boundaries of components of type C}, we see that the renormalizable parameter $v_0 \in \partial \mathcal{U}$ must be post-critically finite.
\end{proof}

\section{Hyperbolic components of type C are quasidisks}\label{sec: hyperbolic components of type C are quasidisks}
In this section, we prove that hyperbolic components of type C are quasidisks.
\begin{lemma}\label{lm: hyperbolic components of type C have disjoint closures}
    Let $\mathcal{U}_1, \mathcal{U}_2$ be two components of $\hat{\mathcal{H}}$. Then $\overline{\mathcal{U}_1} \cap \overline{\mathcal{U}_2}=\varnothing$. 
\end{lemma}
\begin{proof}
    If one of $\mathcal{U}_i$, $i=1,2$ is $\mathcal{H}_0$, we may assume that $\mathcal{U}_2=\mathcal{H}_0$. Let $v \in \overline{\mathcal{U}_1} \cap \overline{\mathcal{H}_0}$. In the view of $\mathcal{U}_1$, by Corollary \ref{cor: dynamics of parameters on the boundaries of hyperbolic components of type C}, $v$ is the landing point of a parameter internal ray $\mathcal{R}_{\mathcal{U}_1}(\theta)$ so that $f^{m}_v(-2v) \in \partial B_v$ is the landing point of $R^v_0(\theta)$, and $-2v \notin \partial B_v$. Moreover, $f_v$ has no parabolic periodic points since the critical value $-2v$ belongs to $J(f_v)$. However, in the view of $\mathcal{H}_0$, $v$ is the landing point of $\mathcal{R}_0(\theta')$ so that $R^v_0(\theta')$ lands at $-2v$. This implies that $-2v \in \partial B_v$, a contradiction. 

    Now suppose that $\mathcal{U}_1$ and $\mathcal{U}_2$ are the components of $\mathcal{H}_{m_1}$ and $\mathcal{H}_{m_2}$ respectively. Let $v \in \partial \mathcal{U}_1 \cap \mathcal{U}_2$. In the view of $\mathcal{U}_1$, $f^{m_1}_v(-2v) \in \partial B_v$ and $f^j_v(-2v) \notin \partial B_v$ for $0 \le j \le m_1-1$. Similarly, $f^{m_2}_v(-2v) \in \partial B_v$ and $f^j_v(-2v) \notin \partial B_v$ for $0 \le j \le m_2-1$ since $v \in \partial \mathcal{U}_2$. Thus $m_1=m_2=m$. 

    By Lemma \ref{lm: k_v is constant}, there are intergers $k_1, k_2 \in \mathbb{N}$ such that $f^{m-1}_v(-2v) \in \partial (B_v+2k_1\pi) \cap \partial (B_v+2k_2\pi)$ since $v$ is a common boundary point of $\mathcal{U}_1$ and $\mathcal{U}_2$. But $B_v+2k_1\pi$ and $B_v+2k_2\pi$ are both mapped to $B_v$ by $f_v$. This implies that $f^{m-1}_v(-2v)$ is a critical point. This is impossible otherwise $f_v$ would have a periodic critical point in $J(f_v)$. 
\end{proof}

\begin{lemma}\label{lm: B_v is a quasidisk}
    Let $\mathcal{U}$ be a component of $\mathcal{H}_m$, $m \ge 1$, $v_0 \in \mathcal{U}$. Then $B_{v_0}$ is a quasidisk.
\end{lemma}
\begin{proof}
    $f_{v_0}$ is hyperbolic with exactly two critical values. Since the critical values are in the different Fatou components, the Fatou components of $f_{v_0}$ are all bounded by \cite{QW25}, Lemma 2.1. By \cite{BFR15}, Theorem 1.4, $B_{v_0}$ is a quasidisk. 
\end{proof}

\begin{lemma}\label{lm: components of H_m are discrete}
    Let $m \ge 1$. Then the components of $\mathcal{H}_m$ are discrete in $\mathbb{C}$. 
\end{lemma}
\begin{proof}
    Fix $R>0$ to be sufficiently large. Define $\mathcal{W}_1$ to be the collection of components of $\bigcup_{n=1}^m \mathcal{H}_n$ which are contained in $\overline{D(0,R)}$, and $\mathcal{W}_2$ to the the collection of components of $\bigcup_{n=1}^m \mathcal{H}_n$, which intersects the complement of $\overline{D(0,R)}$.
    
    First we show that $\mathcal{W}_1$ contains finitely many components. Suppose by contradiction that there are infinitely many components $\mathcal{U}_k$ of $\bigcup_{n=1}^m$ which are contained in $\overline{D(0,R)}$. We may assume that $\mathcal{U}_k$ are all components of $\mathcal{H}_n$ for some $1 \le n \le m$. Let $v_k$ be the center of $\mathcal{U}_k$. Then $v_k$ is the solution to $f^n_{v_k}(-2v_k)=0$. However, this would imply that the analytic function $G_n(v)=f^n_v(-2v)$ has infinitely many zeros in the compact set $\overline{D(0,R)}$, so that $G_n \equiv 0$, a contradiction.
    
    Then we show that the elements of $\mathcal{W}_2$ are discrete. Indeed, let $G_m(v)=f^m_v(-2v)$. For $v \in \left(\bigcup_{n=1}^m \mathcal{H}_n\right) \setminus \overline{D(0,R)}$, $\mathrm{diam} \; B_v \le 2\sqrt{2}/R$ by Lemma \ref{lm: diameter of B_v tends to 0}. Thus $\left(\bigcup_{n=1}^m \mathcal{H}_n\right) \setminus \overline{D(0,R)} \subset G^{-1}_m(D(0,2\sqrt{2}/R))$. By the Open Mapping Theorem, the components of $G^{-1}_m(D(0,2\sqrt{2}/R))$ can only accumulate at $\infty$ or $\partial D(0,R)$. 

    If $\{U_k\}_{k \ge 1}$ is a sequence of components of $G^{-1}_m(D(0,2\sqrt{2}/R))$ accumulating at $\partial D(0,R)$, then there is an integer $K$ such that $U_k \subset D(0,R+1)$ for $k \ge K$. As shown above, there are only finitely many components of $\bigcup_{n=1}^m \mathcal{H}_n$ which are entirely contained in $D(0,R+1)$. Hence $\left(\bigcup_{n=1}^m \mathcal{H}_n\right) \cap U_k$, $k \ge K$, are contained finitely many components of $\bigcup_{n=1}^m \mathcal{H}_n$. Therefore, the elements of $\mathcal{W}_2$ can only accumulate at $\infty$ so that they are discrete.  
\end{proof}

Now let $\mathcal{U}$ be a component of $\mathcal{H}_m$, $m \ge 1$. By Lemma \ref{lm: boundness of components of type C}, there is $R>0$ such that $D(0,R)$ contains $\mathcal{U}$. Let $W$ be a connected subset of $D(0,2R)$ containing $\mathcal{U}$ such that $\overline{W} \cap \overline{\mathcal{V}} =\varnothing$ as long as $\mathcal{V}$ is a component of $\mathcal{H}_n$, $0 \le n \le m$, distinct from $\mathcal{U}$. Such $W$ exists since the components of $\bigcup_{n=0}^m \mathcal{H}_n$ have disjoint closures and are discrete by Lemma \ref{lm: hyperbolic components of type C have disjoint closures} and \ref{lm: components of H_m are discrete}. 

\begin{lemma}\label{lm: holomorphic motion of the immediate basin}
    Let $v_0$ be the center of $\mathcal{U}$. There is a holomorphic motion $\xi: W \times \overline{B_{v_0}} \to \mathbb{C}$.
\end{lemma}
\begin{proof}
    Since $\overline{W} \cap \overline{\mathcal{H}_0}=\varnothing$, for every $v \in W$, $B_{v}$ is bounded by Corollary \ref{cor: boundedness of B_v}. And the B\"ottcher mapping is well-defined on the entire immediate basin $B_v$. So there is a conformal isomorphism $\phi_v: B_v \to \mathbb{D}$. Define $\xi(v,z)=\phi^{-1}_v \circ \phi_{v_0}(z)$. Then it is easy to verify that $\xi: W \times B_v \to \mathbb{C}$ is a holomorphic motion. By $\lambda$-lemma (\cite{MSS83}), $\xi$ extends to a holomorphic motion $\xi: W \times \overline{B_v} \to \mathbb{C}$. 
\end{proof}

By Slodkowski's Theorem, the holomorphic mothion $\xi$ defined in Lemma \ref{lm: holomorphic motion of the immediate basin} extends to a holomorphic motion $\xi: W \times \hat{\mathbb{C}} \to \hat{\mathbb{C}}$. Define 
\begin{equation}\label{eq: Xi}
    \begin{split}
        \Xi: W & \to \mathbb{C} \\
        v & \mapsto \xi^{-1}_v(f^m_v(-2v))
    \end{split}.
\end{equation}
Choose a domain $V \Subset W$, which contains $\overline{\mathcal{U}}$. Similar to Lemma \ref{lm: non-degenerated para-annuli} (also \cite{WQ19}, Lemma 3.3), there is $K \ge 1$ such that $\Xi$ given by \eqref{eq: Xi} is a $K$-quasi-regular mapping on $V$. 

\begin{proof}[Proof of Theorem \ref{thm: hyperbolic components of type C are quasidisks}]
    We will show that $\Xi: \overline{\mathcal{U}} \to \overline{B_{v_0}}$ is a $K$-quasiconformal mapping. Then $\mathcal{U}$ is a quasidisk since $B_{v_0}$ is a quasidisk. 

    For $v \in \mathcal{U}$, $\Xi(v)=\phi^{-1}_{v_0} \circ \phi_v(f^m_v(-2v))=\phi^{-1}_{v_0} \circ \Phi_{\mathcal{U}}(v) \in B_{v_0}$. Since $\phi^{-1}_{v_0}: \mathbb{D} \to B_{v_0}$ and $\Phi_{\mathcal{U}}: \mathcal{U} \to \mathbb{D}$ are both conformal isomorphisms, $\Xi: \mathcal{U} \to B_{v_0}$ is a conformal isomorphism. Since $\mathcal{U}$ and $B_{v_0}$ are both Jordan domains (Corollary \ref{cor: local connectivity of B_v} and \ref{cor: hyperbolic components of type C are Jordan domains}), $\Xi: \overline{\mathcal{U}} \to \overline{B_{v_0}}$ is a homeomorphism by Carath\'eodory's Theorem. 

    Then we show that $\Xi^{-1}(\overline{B}_{v_0})=\overline{\mathcal{U}}$. Let $v \in \Xi^{-1}(\overline{B}_{v_0})$. Then $\Xi(v)=\phi^{-1}_{v_0} \circ \phi_v(f^m_v(-2v)) \in B_{v_0}$. This implies that $f^m_v(-2v) \in B_v$. By the definition, $v \in \bigcup_{n=0}^m \mathcal{H}_n$. Since $W$ intersects no components of $\bigcup_{n=0}^m \mathcal{H}_n$ other than $\mathcal{U}$, $v \in \mathcal{U}$. Let $v \in W$ with $\Xi(v) \in \partial B_{v_0}$. Then by the Open Mapping Theorem, for every open neighborhood $N$ of $v$, there are $v_1, v_2 \in N$ such that $\Xi(v_1) \in B_{v_0}$, $\Xi(v_2) \notin B_{v_0}$. This implies that $v_1 \in \mathcal{U}$ and $v_2 \notin \mathcal{U}$. Thus $v \in \partial \mathcal{U}$. 

    Finally, similar to the proof of Lemma \ref{lm: copy of Mandelbrot set}, $\Xi$ is bijective on the neighborhood $V$ of $\overline{\mathcal{U}}$ so that $\Xi: V \to \Xi(V)$ is a $K$-quasiconformal mapping. $B_{v_0}$ is a quasidisk by Lemma \ref{lm: B_v is a quasidisk}. Hence $\mathcal{U}=\Xi^{-1}(B_{v_0})$ is a quasidisk. 
\end{proof}

\bibliographystyle{plain}
\bibliography{Hyperbolic-components-of-cosine-family.bib}

@book {DH84,
    AUTHOR = {Douady, A. and Hubbard, J. H.},
     TITLE = {\'Etude dynamique des polyn\^omes complexes. {P}artie {I}},
    SERIES = {Publications Math\'ematiques d'Orsay [Mathematical
              Publications of Orsay]},
    VOLUME = {84-2},
 PUBLISHER = {Universit\'e{} de Paris-Sud, D\'epartement de Math\'ematiques,
              Orsay},
      YEAR = {1984},
     PAGES = {75}
}

@article {GS98,
    AUTHOR = {Graczyk, J. and Smirnov, S.},
     TITLE = {Collet, {E}ckmann and {H}\"older},
   JOURNAL = {Invent. Math.},
  FJOURNAL = {Inventiones Mathematicae},
    VOLUME = {133},
      YEAR = {1998},
    NUMBER = {1},
     PAGES = {69--96}
}

@article {BFR15,
    AUTHOR = {Bergweiler, W. and Fagella, N. and Rempe-Gillen,
              L.},
     TITLE = {Hyperbolic entire functions with bounded {F}atou components},
   JOURNAL = {Comment. Math. Helv.},
  FJOURNAL = {Commentarii Mathematici Helvetici. A Journal of the Swiss
              Mathematical Society},
    VOLUME = {90},
      YEAR = {2015},
    NUMBER = {4},
     PAGES = {799--829}
}

@article {KSS07,
    AUTHOR = {Kozlovski, O. and Shen, W. and van Strien, S.},
     TITLE = {Density of hyperbolicity in dimension one},
   JOURNAL = {Ann. of Math. (2)},
  FJOURNAL = {Annals of Mathematics. Second Series},
    VOLUME = {166},
      YEAR = {2007},
    NUMBER = {1},
     PAGES = {145--182}
}

@article {BH92,
    AUTHOR = {Branner, B. and Hubbard, J. H.},
     TITLE = {The iteration of cubic polynomials. {II}. {P}atterns and
              parapatterns},
   JOURNAL = {Acta Math.},
  FJOURNAL = {Acta Mathematica},
    VOLUME = {169},
      YEAR = {1992},
    NUMBER = {3-4},
     PAGES = {229--325}
}

@article {EL92,
    AUTHOR = {Er\"emenko, A. \`E. and Lyubich, M. Yu.},
     TITLE = {Dynamical properties of some classes of entire functions},
   JOURNAL = {Ann. Inst. Fourier (Grenoble)},
  FJOURNAL = {Universit\'e{} de Grenoble. Annales de l'Institut Fourier},
    VOLUME = {42},
      YEAR = {1992},
    NUMBER = {4},
     PAGES = {989--1020}
}

@inbook{RS08, 
  place={Cambridge}, 
  series={London Mathematical Society Lecture Note Series}, 
  title={Escaping points of the cosine family}, 
  booktitle={Transcendental Dynamics and Complex Analysis}, 
  publisher={Cambridge University Press}, 
  author={Rottenfusser, G. and Schleicher, D.}, 
  year={2008}, 
  pages={396–424}, 
  collection={London Mathematical Society Lecture Note Series}
}

@article {BR20,
    AUTHOR = {Benini, A. and Rempe, L.},
     TITLE = {A landing theorem for entire functions with bounded
              post-singular sets},
   JOURNAL = {Geom. Funct. Anal.},
  FJOURNAL = {Geometric and Functional Analysis},
    VOLUME = {30},
      YEAR = {2020},
    NUMBER = {6},
     PAGES = {1465-1530}
}

@article {DH85,
    AUTHOR = {Douady, A. and Hubbard, J. H.},
     TITLE = {On the dynamics of polynomial-like mappings},
   JOURNAL = {Ann. Sci. \'Ecole Norm. Sup. (4)},
  FJOURNAL = {Annales Scientifiques de l'\'Ecole Normale Sup\'erieure.
              Quatri\`eme S\'erie},
    VOLUME = {18},
      YEAR = {1985},
    NUMBER = {2},
     PAGES = {287--343}
}

@book {McM94,
    AUTHOR = {McMullen, C. T.},
     TITLE = {Complex dynamics and renormalization},
    SERIES = {Annals of Mathematics Studies},
    VOLUME = {135},
 PUBLISHER = {Princeton University Press, Princeton, NJ},
      YEAR = {1994}
}

@book {Why42,
    AUTHOR = {Whyburn, G.},
     TITLE = {Analytic {T}opology},
    SERIES = {American Mathematical Society Colloquium Publications},
    VOLUME = {Vol. 28},
 PUBLISHER = {American Mathematical Society, New York},
      YEAR = {1942},
     PAGES = {x+278}
}

@article {BM02,
    AUTHOR = {Bergweiler, W. and Morosawa, S.},
     TITLE = {Semihyperbolic entire functions},
   JOURNAL = {Nonlinearity},
  FJOURNAL = {Nonlinearity},
    VOLUME = {15},
      YEAR = {2002},
    NUMBER = {5},
     PAGES = {1673-1684}
}

@article {RS09,
    AUTHOR = {Rempe, L. and Schleicher, D.},
     TITLE = {Bifurcations in the space of exponential maps},
   JOURNAL = {Invent. Math.},
  FJOURNAL = {Inventiones Mathematicae},
    VOLUME = {175},
      YEAR = {2009},
    NUMBER = {1},
     PAGES = {103--135}
}

@book {CG93,
    AUTHOR = {Carleson, L. and Gamelin, T. W.},
     TITLE = {Complex dynamics},
    SERIES = {Universitext: Tracts in Mathematics},
 PUBLISHER = {Springer-Verlag, New York},
      YEAR = {1993}
}

@article {Roe07,
    AUTHOR = {Roesch, P.},
     TITLE = {Hyperbolic components of polynomials with a fixed critical
              point of maximal order},
   JOURNAL = {Ann. Sci. \'Ecole Norm. Sup. (4)},
  FJOURNAL = {Annales Scientifiques de l'\'Ecole Normale Sup\'erieure.
              Quatri\`eme S\'erie},
    VOLUME = {40},
      YEAR = {2007},
    NUMBER = {6},
     PAGES = {901-949}
}

@article {Tian11,
    AUTHOR = {Tian, T.},
     TITLE = {Parameter rays for the cosine family},
   JOURNAL = {J. Fudan Univ. Nat. Sci.},
  FJOURNAL = {Fudan University. Journal. Natural Science. Fudan Xuebao.
              Ziran Kexue Ban},
    VOLUME = {50},
      YEAR = {2011},
    NUMBER = {1},
     PAGES = {10-22}
}

@book {Sch99,
    AUTHOR = {Schleicher, D.},
     TITLE = {On the {D}ynamics of {I}terated {E}xponential {M}aps},
      NOTE = {Thesis (Ph.D.)--Technische Universit\"at M\"unchen (Germany)},
 PUBLISHER = {ProQuest LLC, Ann Arbor, MI},
      YEAR = {1999},
     PAGES = {207}
}

@article {MSS83,
    AUTHOR = {Ma\~n\'e, R. and Sad, P. and Sullivan, D.},
     TITLE = {On the dynamics of rational maps},
   JOURNAL = {Ann. Sci. \'Ecole Norm. Sup. (4)},
  FJOURNAL = {Annales Scientifiques de l'\'Ecole Normale Sup\'erieure.
              Quatri\`eme S\'erie},
    VOLUME = {16},
      YEAR = {1983},
    NUMBER = {2},
     PAGES = {193-217}
}

@article {QRWY15,
    AUTHOR = {Qiu, W. and Roesch, P. and Wang, X. and Yin, Y.},
     TITLE = {Hyperbolic components of {M}c{M}ullen maps},
   JOURNAL = {Ann. Sci. \'Ec. Norm. Sup\'er. (4)},
  FJOURNAL = {Annales Scientifiques de l'\'Ecole Normale Sup\'erieure.
              Quatri\`eme S\'erie},
    VOLUME = {48},
      YEAR = {2015},
    NUMBER = {3},
     PAGES = {703-737}
}

@misc{QW25,
      title={Bounded Fatou components of cosine functions}, 
      author={Qiu, W. and Wang, L.},
      year={2025},
      eprint={2511.20010},
      archivePrefix={arXiv},
      primaryClass={math.CV},
      url={https://arxiv.org/abs/2511.20010}
}

@incollection {Roe00,
    AUTHOR = {Roesch, P.},
     TITLE = {Holomorphic motions and puzzles (following {M}. {S}hishikura)},
 BOOKTITLE = {The {M}andelbrot set, theme and variations},
    SERIES = {London Math. Soc. Lecture Note Ser.},
    VOLUME = {274},
     PAGES = {117--131},
 PUBLISHER = {Cambridge Univ. Press, Cambridge},
      YEAR = {2000}
}

@article{WQ19,
  author = {Wang, Y. and Qiu, W.},
  journal = {Sci. Sin. Math.},
  FJOURNAL = {SCIENTIA SINICA Mathematica},
  number = {10},
  title = {Capture components are quasidisks: {C}ubic polynomials with a
critical fixed point (in {C}hinese)},
  volume = {49},
  year = {2019},
  pages = {1431-1438}
}

@incollection {Mil09,
    AUTHOR = {Milnor, J.},
     TITLE = {Cubic polynomial maps with periodic critical orbit. {I}},
 BOOKTITLE = {Complex dynamics},
     PAGES = {333--411},
 PUBLISHER = {A K Peters, Wellesley, MA},
      YEAR = {2009},
      ISBN = {978-1-56881-450-6},
   MRCLASS = {37F10 (37F50)},
  MRNUMBER = {2508263},
MRREVIEWER = {Peter\ Ha\"issinsky},
       DOI = {10.1201/b10617-13},
       URL = {https://doi.org/10.1201/b10617-13},
}

@incollection {Mil12,
    AUTHOR = {Milnor, J.},
     TITLE = {Hyperbolic components},
 BOOKTITLE = {Conformal dynamics and hyperbolic geometry},
    SERIES = {Contemp. Math.},
    VOLUME = {573},
     PAGES = {183--232},
      NOTE = {With an appendix by A. Poirier},
 PUBLISHER = {Amer. Math. Soc., Providence, RI},
      YEAR = {2012},
      ISBN = {978-0-8218-5348-1},
   MRCLASS = {37Fxx},
  MRNUMBER = {2964079},
       DOI = {10.1090/conm/573/11428},
       URL = {https://doi.org/10.1090/conm/573/11428},
}

@article {Ree90,
    AUTHOR = {Rees, M.},
     TITLE = {Components of degree two hyperbolic rational maps},
   JOURNAL = {Invent. Math.},
  FJOURNAL = {Inventiones Mathematicae},
    VOLUME = {100},
      YEAR = {1990},
    NUMBER = {2},
     PAGES = {357--382},
      ISSN = {0020-9910,1432-1297},
   MRCLASS = {58F15 (30D05)},
  MRNUMBER = {1047139},
MRREVIEWER = {Artur\ Oscar\ Lopes},
       DOI = {10.1007/BF01231191},
       URL = {https://doi.org/10.1007/BF01231191},
}

@article {BH88,
    AUTHOR = {Branner, B. and Hubbard, J. H.},
     TITLE = {The iteration of cubic polynomials. {I}. {T}he global topology
              of parameter space},
   JOURNAL = {Acta Math.},
  FJOURNAL = {Acta Mathematica},
    VOLUME = {160},
      YEAR = {1988},
    NUMBER = {3-4},
     PAGES = {143--206},
      ISSN = {0001-5962,1871-2509},
   MRCLASS = {30D05 (58F08)},
  MRNUMBER = {945011},
MRREVIEWER = {I.\ N.\ Baker},
       DOI = {10.1007/BF02392275},
       URL = {https://doi.org/10.1007/BF02392275},
}

@article {Ree92,
    AUTHOR = {Rees, M.},
     TITLE = {A partial description of parameter space of rational maps of
              degree two. {I}},
   JOURNAL = {Acta Math.},
  FJOURNAL = {Acta Mathematica},
    VOLUME = {168},
      YEAR = {1992},
    NUMBER = {1-2},
     PAGES = {11--87},
      ISSN = {0001-5962,1871-2509},
   MRCLASS = {58F23 (30D05)},
  MRNUMBER = {1149864},
MRREVIEWER = {Feliks\ Przytycki},
       DOI = {10.1007/BF02392976},
       URL = {https://doi.org/10.1007/BF02392976},
}

@article {Eps00,
    AUTHOR = {Epstein, A. L.},
     TITLE = {Bounded hyperbolic components of quadratic rational maps},
   JOURNAL = {Ergodic Theory Dynam. Systems},
  FJOURNAL = {Ergodic Theory and Dynamical Systems},
    VOLUME = {20},
      YEAR = {2000},
    NUMBER = {3},
     PAGES = {727--748},
      ISSN = {0143-3857,1469-4417},
   MRCLASS = {37F45 (30D05 35F25 37F10)},
  MRNUMBER = {1764925},
MRREVIEWER = {Hartje\ Kriete},
       DOI = {10.1017/S0143385700000390},
       URL = {https://doi.org/10.1017/S0143385700000390},
}

@article {DFJ02,
    AUTHOR = {Devaney, R. L. and Fagella, N. and Jarque, X.},
     TITLE = {Hyperbolic components of the complex exponential family},
   JOURNAL = {Fund. Math.},
  FJOURNAL = {Fundamenta Mathematicae},
    VOLUME = {174},
      YEAR = {2002},
    NUMBER = {3},
     PAGES = {193--215},
      ISSN = {0016-2736,1730-6329},
   MRCLASS = {37F10 (30D05)},
  MRNUMBER = {1924998},
MRREVIEWER = {Peter\ Ha\"issinsky},
       DOI = {10.4064/fm174-3-1},
       URL = {https://doi.org/10.4064/fm174-3-1},
}

@book {Fau92,
    AUTHOR = {Faught, J. D.},
     TITLE = {Local connectivity in a family of cubic polynomials},
      NOTE = {Thesis (Ph.D.)--Cornell University},
 PUBLISHER = {ProQuest LLC, Ann Arbor, MI},
      YEAR = {1992},
     PAGES = {110},
   MRCLASS = {99-05},
  MRNUMBER = {2688243},
       URL = {http://gateway.proquest.com/openurl?url_ver=Z39.88-2004&rft_val_fmt=info:ofi/fmt:kev:mtx:dissertation&res_dat=xri:pqdiss&rft_dat=xri:pqdiss:9300808},
}

@article {Wan21,
    AUTHOR = {Wang, X.},
     TITLE = {Hyperbolic components and cubic polynomials},
   JOURNAL = {Adv. Math.},
  FJOURNAL = {Advances in Mathematics},
    VOLUME = {379},
      YEAR = {2021},
     PAGES = {Paper No. 107554, 42},
      ISSN = {0001-8708,1090-2082},
   MRCLASS = {37F46 (37F10 37F15)},
  MRNUMBER = {4199274},
MRREVIEWER = {Yan\ Gao},
       DOI = {10.1016/j.aim.2020.107554},
       URL = {https://doi.org/10.1016/j.aim.2020.107554},
}

@article {KK97,
    AUTHOR = {Keen, L. and Kotus, J.},
     TITLE = {Dynamics of the family {$\lambda\tan z$}},
   JOURNAL = {Conform. Geom. Dyn.},
  FJOURNAL = {Conformal Geometry and Dynamics. An Electronic Journal of the
              American Mathematical Society},
    VOLUME = {1},
      YEAR = {1997},
     PAGES = {28--57},
      ISSN = {1088-4173},
   MRCLASS = {58F23 (30D05)},
  MRNUMBER = {1463839},
MRREVIEWER = {I.\ N.\ Baker},
       DOI = {10.1090/S1088-4173-97-00017-9},
       URL = {https://doi.org/10.1090/S1088-4173-97-00017-9},
}

@incollection {Yua12,
    AUTHOR = {Yuan, S.},
     TITLE = {Parameter plane of a family of meromorphic functions with two
              asymptotic values},
 BOOKTITLE = {Conformal dynamics and hyperbolic geometry},
    SERIES = {Contemp. Math.},
    VOLUME = {573},
     PAGES = {245--256},
 PUBLISHER = {Amer. Math. Soc., Providence, RI},
      YEAR = {2012},
      ISBN = {978-0-8218-5348-1},
   MRCLASS = {37F15 (30D05 30F45 32A20 37F10 37F45)},
  MRNUMBER = {2964081},
MRREVIEWER = {Kevin\ M.\ Pilgrim},
       DOI = {10.1090/conm/573/11394},
       URL = {https://doi.org/10.1090/conm/573/11394},
}

@incollection {KY06,
    AUTHOR = {Keen, L. and Yuan, S.},
     TITLE = {Parabolic perturbation of the family {$\lambda\tan z$}},
 BOOKTITLE = {Complex dynamics},
    SERIES = {Contemp. Math.},
    VOLUME = {396},
     PAGES = {115--128},
 PUBLISHER = {Amer. Math. Soc., Providence, RI},
      YEAR = {2006},
      ISBN = {0-8218-3625-0},
   MRCLASS = {37F10 (30D05)},
  MRNUMBER = {2209090},
MRREVIEWER = {Walter\ Bergweiler},
       DOI = {10.1090/conm/396/07398},
       URL = {https://doi.org/10.1090/conm/396/07398},
}

@incollection {KK00,
    AUTHOR = {Keen, L. and Kotus, J.},
     TITLE = {On period doubling phenomena and {S}harkovskii type ordering
              for the family {$\lambda\tan(z)$}},
 BOOKTITLE = {Value distribution theory and complex dynamics ({H}ong {K}ong,
              2000)},
    SERIES = {Contemp. Math.},
    VOLUME = {303},
     PAGES = {51--78},
 PUBLISHER = {Amer. Math. Soc., Providence, RI},
      YEAR = {2002},
      ISBN = {0-8218-2980-7},
   MRCLASS = {37F10 (30D05 37G99)},
  MRNUMBER = {1943527},
MRREVIEWER = {Walter\ Bergweiler},
       DOI = {10.1090/conm/303/05240},
       URL = {https://doi.org/10.1090/conm/303/05240},
}

@article {McM87,
    AUTHOR = {McMullen, C. T.},
     TITLE = {Area and {H}ausdorff dimension of {J}ulia sets of entire
              functions},
   JOURNAL = {Trans. Amer. Math. Soc.},
  FJOURNAL = {Transactions of the American Mathematical Society},
    VOLUME = {300},
      YEAR = {1987},
    NUMBER = {1},
     PAGES = {329--342},
      ISSN = {0002-9947,1088-6850},
   MRCLASS = {30D05 (58F08 58F20)},
  MRNUMBER = {871679},
MRREVIEWER = {I.\ N.\ Baker},
       DOI = {10.2307/2000602},
       URL = {https://doi.org/10.2307/2000602},
}

@article {Lyu97,
    AUTHOR = {Lyubich, M. Yu.},
     TITLE = {Dynamics of quadratic polynomials. {I}, {II}},
   JOURNAL = {Acta Math.},
  FJOURNAL = {Acta Mathematica},
    VOLUME = {178},
      YEAR = {1997},
    NUMBER = {2},
     PAGES = {185--247, 247--297},
      ISSN = {0001-5962,1871-2509},
   MRCLASS = {58F23 (30C10 30D05)},
  MRNUMBER = {1459261},
MRREVIEWER = {Grzegorz\ \'Swi\polhk atek},
       DOI = {10.1007/BF02392694},
       URL = {https://doi.org/10.1007/BF02392694},
}

@article {NP22,
    AUTHOR = {Nie, H. and Pilgrim, K. M.},
     TITLE = {Bounded hyperbolic components of bicritical rational maps},
   JOURNAL = {J. Mod. Dyn.},
  FJOURNAL = {Journal of Modern Dynamics},
    VOLUME = {18},
      YEAR = {2022},
     PAGES = {533--553},
      ISSN = {1930-5311,1930-532X},
   MRCLASS = {37F20 (37F44 37P50)},
  MRNUMBER = {4505953},
MRREVIEWER = {Alfredo\ Poirier},
       DOI = {10.3934/jmd.2022016},
       URL = {https://doi.org/10.3934/jmd.2022016},
}

@article {Fat20,
    AUTHOR = {Fatou, P.},
     TITLE = {Sur les \'equations fonctionnelles},
   JOURNAL = {Bull. Soc. Math. France},
  FJOURNAL = {Bulletin de la Soci\'et\'e{} Math\'ematique de France},
    VOLUME = {48},
      YEAR = {1920},
     PAGES = {33--94},
      ISSN = {0037-9484},
   MRCLASS = {99-04},
  MRNUMBER = {1504792},
       URL = {http://www.numdam.org/item?id=BSMF_1920__48__33_0},
}

\end{document}